\documentclass[10pt]{amsart}

\usepackage{geometry}                % See geometry.pdf to learn the layout options. There are lots.
\usepackage{graphicx}
\usepackage{amsmath}
\usepackage{amsthm}
\usepackage{amssymb}
\usepackage{epstopdf}
\usepackage{color}
\usepackage{dsfont}
\usepackage{todonotes}
\usepackage[sort,nocompress]{cite}

\def\d{\delta}

\newtheorem{thm}{Theorem}
\newtheorem{prop}[thm]{Proposition}

\newtheorem{remark}{Remark}
\newtheorem{assump}{}

\newtheorem{definition}{Definition}

\newcommand{\mathscr}{\mathcal}

\newcommand{\ito}{It\=o}
\newcommand{\college}{\textit{Coll\`ege de France}}
\newcommand{\frechet}{Fr\'echet }

\newcommand{\rref}[1]{(\ref{#1})}

\newcommand{\Ind}{\mathds{1}}

\newcommand{\bi}{\begin{itemize}}

\newcommand{\ei}{\end{itemize}}
\newcommand{\be}{\begin{equation}}
\newcommand{\ee}{\end{equation}}

\renewcommand{\l}{\left}
\renewcommand{\r}{\right}
\newcommand{\la}{\l\langle}
\newcommand{\ra}{\r\rangle}

\newcommand{\R}{\mathbb{R}}

\newcommand{\J}{\mathcal{J}}
\newcommand{\Je}{\mathcal{J}^\eps}
\newcommand{\U}{\mathcal{U}}
\renewcommand{\P}{\mathscr{P}}

\newcommand{\Mt}[1]{\mathscr{L}^2_{#1}(\Ptwo)}
\newcommand{\MM}[2]{\mathscr{M}([#1,#2],\R)}
\renewcommand{\H}{\mathscr{H}}
\newcommand{\W}{\mathscr{W}}

\renewcommand{\L}{\mathscr{L}}

\newcommand{\Ptwo}{\mathscr{P}_{2}(\mathbb{R})}
\newcommand{\Htwo}{\H^{2}([0,T];\mathbb{R})}

\newcommand{\HtwotT}{\H^{2}([\tau,T];\mathbb{R})}

\newcommand{\Ltwo}{\mathscr{L}^{2}}

\newcommand{\tF}{\tilde{\F}}
\newcommand{\tP}{\tilde{\mathbb{P}}}

\newcommand{\F}{\mathscr{F}}
\newcommand{\G}{\mathscr{G}}

\newcommand{\Law}{\mathcal{L}}
\newcommand{\Laww}[1]{\mathcal{L}(#1)}
\newcommand{\CP}{\mc{M}([0,T];\Ptwo)}

\renewcommand{\t}{\tilde}
\newcommand{\h}{\hat}

\newcommand{\ld}{\lambda}

\newcommand{\sg}{\sigma}
\newcommand{\tw}{\t{\omega}}
\newcommand{\bal}{\bar{\alpha}}
\newcommand{\ha}{\hat{\alpha}}

\newcommand{\hal}{\hat{\alpha}}

\newcommand{\vp}{\varphi}
\newcommand{\mc}{\mathcal}
\newcommand{\mb}{\mathbb}
\newcommand{\mbb}{\mathbb}

\newcommand{\eps}{\varepsilon}

\newcommand{\tdW}{d\tilde{W}}
\newcommand{\tW}{\tilde{W}}

\newcommand{\tZ}{\tilde{Z}}
\newcommand{\tC}{\tilde{C}}
\newcommand{\tQ}{\tilde{Q}}

\newcommand{\tE}{\tilde{\mb{E}}}

\newcommand{\hE}{\hat{\mb{E}}^0}
\newcommand{\hU}{\hat{U}}
\newcommand{\hV}{\hat{V}}

\newcommand{\hX}{\hat{X}}

\newcommand{\hY}{\hat{Y}}

\newcommand{\hw}{\hat{\omega}^0}

\newcommand{\Ve}{\mathcal{V}^{\eps}}
\newcommand{\Ue}{\mathcal{U}^{\eps}}
\newcommand{\Uo}{\mathcal{U}^{0}}
\newcommand{\ue}{u^{\eps}}
\newcommand{\me}{m^{\eps}}

\newcommand{\ve}{v^{\eps}}

\newcommand{\hae}{\hat{\alpha}^\eps}

\newcommand{\Vo}{\mathcal{V}^{0}}
\newcommand{\uo}{u^{0}}
\newcommand{\mo}{m^{0}}

\newcommand{\du}{ \delta^{u,\eps}}
\newcommand{\dm}{ \delta^{m,\eps}}

\newcommand{\dv}{ \delta^{v}}
\newcommand{\duo}{ \delta^{u}}
\newcommand{\dmo}{ \delta^{m}}

\newcommand{\bU}{\bar{U}}
\newcommand{\bV}{\bar{V}}
\newcommand{\bQ}{\bar{Q}}
\newcommand{\tbQ}{\t{\bar{Q}}}

\newcommand{\hmbE}{\hat{\mathbb{E}}^0}

\newcommand{\Xe}{X^{\eps}}
\newcommand{\Ye}{Y^{\eps}}
\newcommand{\Ze}{Z^{\eps}}
\newcommand{\tZe}{\tilde{Z}^{\eps}}
\newcommand{\Xo}{X^{0}}
\newcommand{\Yo}{Y^{0}}

\newcommand{\dXe}{\Delta X^{\eps}}
\newcommand{\dYe}{\Delta Y^{\eps}}
\newcommand{\dZe}{\Delta Z^{\eps}}
\newcommand{\dtZe}{\Delta \tilde{Z}^{\eps}}

\renewcommand{\d}[1]{\partial_{#1}}
\newcommand{\dd}[1]{\partial^2_{#1}}

\title{Asymptotic Analysis of Mean Field Games with Small Common Noise}

\author{Saran Ahuja, Weiluo Ren, Tzu-Wei Yang} 
\thanks{Department of Mathematics, Stanford University, Sloan Hall, Stanford, CA 94305 ({\tt ssunny@stanford.edu})}
\thanks{Department of Mathematics, Stanford University, Sloan Hall, Stanford, CA 94305 ({\tt weiluo@stanford.edu})}
\thanks{School of Mathematics, University of Minnesota, Minneapolis, MN 55455
({\tt yangx953@umn.edu})}
\begin{document}

\begin{abstract}
In this paper, we consider a mean field game (MFG) model perturbed by small common noise. Our goal is to give an approximation of the Nash equilibrium strategy of this game using a solution from the original no common noise MFG whose solution can be obtained through a coupled system of partial differential equations. We characterize the first order approximation via linear mean-field forward-backward stochastic differential equations whose solution is a centered Gaussian process with respect to the common noise. The first order approximate strategy can be described as follows: at time $t \in [0,T]$, applying the original MFG optimal strategy for a sub game over $[t,T]$ with the initial being the current state and distribution. We then show that this strategy gives an approximate Nash equilibrium of order $\epsilon^2$. 
\end{abstract}

\maketitle

%\begin{keywords} 
%mean field games, common noise, small common noise, approximation, perturbation, forward-backward stochastic differential equations.
%\end{keywords}
%
%\begin{AMS}
%93E20, 60H30, 60H10, 49N70, 49J99.
%\end{AMS}

\pagestyle{myheadings}
\thispagestyle{plain}

\section{Introduction}
Mean field game (MFG) is a limit model for a stochastic differential game with large number of players,  symmetric cost functions, and interactions of mean-field type. Specifically, each player optimizes a control problem whose state process and cost functions depend not only on their own state and control but also on other players' decision through the empirical distribution of their states. Under certain independence assumption, considering the control problem at the asymptotic regime can reduce this high-dimensional complex interacting system to a fully-coupled forward-backward partial differential equations (PDEs). The solution of this system can then be used to approximate the Nash equilibrium solution of the finite player games. This novel idea was first proposed by Lasry and Lions \cite{lasry2006, lasry2006_2,lasry2007} and independently from an engineering community by Caines, Huang, and Malham\'e \cite{huang2007}. 

The MFG problem, with linear-convex setting as will be considered in this paper, is defined as follows;
\begin{equation}\label{mfg_nc}\begin{cases}
    \alpha^* \in \arg\min_{\alpha \in \mc{A}} \mb{E} \l[ \int_0^T \frac{\alpha^2}{2}dt + g(X^\alpha_T,m_T)\r]\\
    dX^\alpha_t = \alpha_t dt + \sigma dW_t  \\
    m_t = \Law(X^{\alpha^*}_t) 
  \end{cases}
\end{equation}

In the past decades, active research has been done on MFG model with tremendous progress in many directions. See \cite{gomes2013survey} for a brief survey and \cite{bensoussan2013} for a more extensive reference. Many extensions of \eqref{mfg_nc} has been considered including a model with major/minor players \cite{bensoussan2014mean,carmona2014majorminor,huang2010large,nguyen2012linear,nourian2013} and a convergence from finite player games to MFG \cite{bardi2014,fischer2014,gomes2012cont,lacker2014general}. An important extension that has gained a lot of attention is the model with common noise, which is a common Brownian motion that occurs in the state process of every players, relaxing the independence assumption assumed in the original model. This type of model comes up frequently and naturally in many applications particularly in finance or economics \cite{moll2013,carmona2013mean,gueant2011,lasry2008application} where each player is subject to some sort of common market factor. MFG with common noise can be formulated as follows; 
\begin{equation}\label{mfg_c}\begin{cases}
    \alpha^* \in \arg\min_{\alpha \in \mc{A}} \mb{E} \l[ \int_0^T \frac{\alpha^2}{2}dt + g(X^\alpha_T,m_T)\r]\\
    dX^\alpha_t = \alpha_t dt + \sigma dW_t + \eps d\tW_t  \\
    m_t = \Law(X^{\alpha^*}_t | \tF_t), \quad \tF_t = \sigma(\tW_s; 0 \leq s \leq t) 
  \end{cases}
\end{equation}

%find a control $\hae \in \mc{A}$, where $\mc{A}$ is the set of admissible controls, such that
%$$ \Je(\hae | m^{\eps,\hae}) \leq \Je(\alpha | m^{\eps,\hae}), \quad \text{for all }\alpha \in \mc{A} $$
%where $\Je(\alpha | m)$ is the cost function corresponding to a control $\alpha = (\alpha_t)_{0 \leq t \leq T}$ given $m = (m_t)_{0 \leq t \leq T}$;
%$$\Je(\alpha | m) \triangleq \mb{E}\l[ \int_0^T \frac{\alpha_t^2}{2}dt + g(X^{\eps,\alpha}_T,m_T)\r],$$
%$X^{\eps,\alpha} = (X^{\eps,\alpha}_t)_{0 \leq t \leq T}$ is the state process; 
%\begin{equation}\label{state_intro} 
%dX^{\eps,\alpha}_t = \alpha_t dt + \sigma dW_t + \eps d\tW_t, \quad X^{\eps,\alpha}_0 = \xi_0,
%\end{equation}
%and $m^{\eps,\alpha} = (m^{\eps,\alpha}_t)_{0 \leq t \leq T}$ is its law conditional on a common noise filtration $\tF_t =  \sigma(\tW_s; 0 \leq s \leq t)$, i.e.
%$$m^{\eps,\alpha}_t = \Law(X^{\eps,\alpha}_t | \tF_t ) $$

The common noise model \eqref{mfg_c} is more complicated than the original MFG \eqref{mfg_nc}. In \eqref{mfg_nc}, the law $m^{0,\alpha}_t$ is expected to be deterministic, so it suffices to seek the optimal strategy along the path $(m^{0,\alpha}_t)_{0 \leq t \leq T}$. This reduces the problem to a finite-dimensional system of forward-backward PDEs. The common noise model, on the other hands, is more complex as the flow of players distribution is stochastic. This means that we need to specify the optimal action for all possible trajectories of the players' distribution which is infinite-dimensional. One way to solve this model is to add $m_t$ as an argument in the value function. This approach leads to the study of the \textit{master} equation which is an infinite-dimensional Hamilton-Jacobi-Bellman (HJB) equation that encapsulates all the information of the MFG. See \cite{bensoussan2014master,carmona2014master,gomes2013survey} for some discussions on this approach. Alternatively, one could follow the same methodology as done by Lasry and Lions. In that case,  instead of a forward-backward PDE, the presence of common noise gives a forward-backward stochastic partial differential equation (SPDE). Lastly, we can also use the probabilistic approach proposed by Carmona and Delarue \cite{carmona2013probabilistic} which formulates the MFG as a mean-field forward-backward stochastic differential equations (FBSDE). The difference between the two approaches to MFG is from the two different approaches to stochastic control problems, namely the Stochastic Maximum Principle (SMP) and the Dynamic Programming Principle (DPP).

Recently, there has been progress in the study of MFG with common noise concerning mostly to well-posedness results. In \cite{ahuja2014,ahuja2016forward}, the existence and uniqueness result of MFG with common noise is proved when the state process is linear and the cost functions satisfy a certain convexity and monotonicity condition. Carmona et. al. \cite{carmona2016} gives the existence and uniqueness result of a \textit{weak} solution under a more general setting. In \cite{bensoussan2014master,carmona2014master}, the master equation was discussed from the perspective of both HJB and probabilistic approaches. Under special circumstances, the common noise model might be explicitly solvable through a transformation which turns the problem to the original no common noise MFG \cite{carmona2013mean,gueant2011,lacker2014translation}. Despite these results, a general common noise model is difficult and impractical to solve numerically or explicitly as it does not enjoy the dimension reduction property as in the case of MFG without common noise. 

%In the original MFG model of Lasry and Lion, we deal with a system of PDEs, hence numerical schemes can be developed to approximate its solution \cite{achdou2010mean}. Intuitively, when there is no common noise, the law $m^{0,\alpha}_t$ is expected to be deterministic, so we only need to find the optimal strategy along the characteristic. A common noise model, on the other hands, is more complex as the flow of players distribution is stochastic. This means that we need to specify the optimal action for all possible trajectory of players' distribution which is infinite-dimensional. We would like to point out that while a common noise model will most likely be intractable, under special circumstances, it might be explicitly solvable through a clever transformation that turns the problem to the original no common noise MFG. See \cite{carmona2013mean,gueant2011} for some examples of this sort and \cite{lacker2014translation} for a more general treatment of what is called a translation-invariant MFG with common noise. 

The goal of this paper is to consider a MFG problem when the common noise is \textit{small} as denoted by the parameter $\eps$ in \rref{mfg_c}. We will refer to this game as $\eps$-MFG. In this set up, we seek an approximate solution using only the information from $\eps=0$ problem, i.e. the original MFG with no common noise or $0$-MFG. If we denote by $(\alpha^\eps_t,X^\eps_t,m^\eps_t)_{0 \leq t \leq T}$ a solution to $\eps$-MFG described above, then we are essentially interested in the following $\eps$-expansion 
\begin{equation}\label{expansion}
 \alpha_t^\eps = \alpha_t^0 + \eps \delta^{\alpha}_t + o(\eps), \qquad X^\eps_t = X^0_t + \eps \delta^{X}_t + o(\eps)
 \end{equation}
Equivalently, we would like to study the limit as $\eps \to 0$ of
\begin{equation}\label{limits} \frac{X^\eps_t - X^0_t}{\eps}, \quad \frac{\alpha_t^\eps-\alpha_t^0}{\eps}
\end{equation}

%Our first main result is proving that the limits above exist and are characterized by a solution to a system of linear FBSDE with terminal condition of a mean-field type. Our assumptions are similar to those in \cite{ahuja2014}, where we assume a linear state process, and convex and weak monotone cost functions, with additional regularity assumptions. Our second main result concerns with the analysis of the solution of this linear FBSDE. As it turns out, under this set of assumptions, the FBSDE describing the limits can be solved explicitly and the solutions \rref{limits} are in fact centered Gaussian processes with respect to the common Brownian motion. Particularly, they are $\tW_t$-path dependent in the form of stochastic integral of $d\tW_t$ with the integrand adapted to the information from $0$-MFG only. We then are able to compute the covariance functions explicitly. 
%
%In addition to finding the limits, the explicit result also enables us to construct a finite dimensional approximation of the $\eps$-MFG solution. In fact, we show that the first order approximate solution at time $t$ is simply an approximation of the $0$-MFG solution for a subgame over $[t,T]$ with initial $\me_t$. We show that the first order effect of the common noise is the spatial shift of the state of the game, particularly the distribution of players, but it does not affect the strategy of the game. However, tracking all possible trajectory of $\me_t$ is infinite dimensional and that is why the first order approximation is necessary to construct a finite dimensional solution.  

This paper contributes mainly to the study of the limit \rref{limits} and the corresponding first order approximate strategy. Our main result is to prove that \rref{limits} converges to a solution of a system of linear mean-field FBSDE whose solution is a centered Gaussian process. While recently there has been much work on the MFG model with common noise or the convergence of $N$-player game to MFG, the asymptotic analysis of small common noise model is new to the best of our knowledges. Our setting and assumptions are similar to those in \cite{ahuja2014} where we assume a linear state process, and convex and weak monotone cost functions, with some additional regularity assumptions. 

In addition to the convergence result, we show that the first order approximate strategy (see \rref{expansion}) gives an approximate Nash equilibrium of order $\eps^2$.  More interestingly, we show that the game-theoretic correction in the optimal strategy due to the existence of (small) common noise is not presented at the first order, and we can simply use the 0-MFG optimal strategy along the trajectory of the stochastic flow $(\me_t)_{0 \leq t \leq T}$. That is, at time $t \in [0,T]$, we solve the sub-game of 0-MFG over $[t,T]$ with the initial being the current distribution $\me_t$. Note that this is different from the 0-MFG solution itself since we use $\me_t$ as the initial at $t$ as opposed to $\mo_t$. 

Our main technical tool is the Stochastic Maximum Principle which turns a MFG problem to a mean-field FBSDE. The linear, convex, monotone assumptions on the MFG leads to a mean-field FBSDE with monotone property. A system of monotone FBSDE is well-studied both in the classical setting \cite{hu1995,peng1999,yong1997continuation} and also recently with mean-field terms\cite{ahuja2014,ahuja2016forward,Bensoussan2015} where probabilistic techniques and standard SDE estimates can be applied. Under this setting, we are able to obtain all the results, the limits and the estimates, in a \textit{strong} sense, namely in $\mathcal{L}^2$, using similar tools.

The paper is organized as follows. In section \ref{mfg_commonnoise}, we consider a general MFG with common noise through the Stochastic Maximum Principle and discuss the well-posedness result as well as existence of the decoupling function all of which will be used in subsequent sections. The main results, namely the asymptotic analysis of $\eps$-MFG, are given in section \ref{asymptotic}. We then discuss a connection between the SMP approach and the DPP approach in section \ref{dpp2}. The Appendices contain the proofs of the main theorems and lemmas as well as discussions on the existence and uniqueness of FBSDE with monotone functionals and the notion of differentiation with respect to a probability measure.

\section{Mean field game with common noise}\label{mfg_commonnoise}
%In this section, we address a general mean field game with common noise including the SMP for $\eps$-MFG, the wellposedness result, markov property, and connection to the HJB approach.

\subsection{Notations and general set up}\label{setup}
%The process $(W_t)_{0 \leq t \leq T}$ is called an \textit{individual} or \textit{idiosyncratic} noise and $(\tW_t)_{0 \leq t \leq T}$ is called a \textit{common} noise.
Fix a terminal time $T > 0$. Let $(W_t)_{0 \leq t \leq T}, (\tW_t)_{0 \leq t \leq T}$ denote two independent Brownian motions on $\R$ defined on a complete filtered probability space $(\Omega, \F, \mb{F}=\{\F_t\}_{0 \leq t \leq T}, \mb{P})$ augmented by a $\mb{P}$-null set. We call $(W_t)_{0 \leq t \leq T}$ the \textit{individual noise} and $(\tW_t)_{0 \leq t \leq T}$ the \textit{common noise}. We assume that $(\Omega,\F,\mb{P})$ is in the form $ (\Omega^0 \times \t{\Omega},\F^0 \otimes \tF,\mb{P}^0 \otimes \tP)$ where $(\t{\Omega},\tF,\tP)$ is the canonical sample space of the common noise $(\tW_t)_{0 \leq t \leq T}$ and that all other randomness, the individual noise and initial, are supported in  $(\Omega^0,\F^0,\mb{P}^0)$.

Let $\Ptwo$ denote the space of Borel probability measure on $\R$ with finite second moment, i.e. all probability measure $\mu$ such that
	$$ \int_\R x^2 d\mu(x) < \infty $$ 
It is a complete separable metric space equipped with a Wasserstein metric defined as
\begin{equation}\label{wass}
 \W_2(m_1,m_2) = \l( \inf_{\gamma \in \Gamma(m_1,m_2)} \int_{\R^2} |x-y|^2 \gamma(dx,dy) \r)^{\frac{1}{2}}
 \end{equation}
where $\Gamma(m_1,m_2)$ denotes the collection of all probability measures on $\R^2$ with marginals $m_1$ and $m_2$.  While we assume one-dimensional Euclidean space for simplicity, all the results in this paper still hold for $\R^d$. Let $\tF^{s}_t$ denote the filtration generated by $\tW_r-\tW_s, s \leq r \leq t$ and we write $\tF_t = \tF^0_t$. Suppose $\G$ is a sub $\sigma$-algebra of $\F$ and $\mb{G} = \{ \G_t \}_{0 \leq t \leq T}$ is a sub filtration of $\mb{F}$, then let $\Ltwo_{\G}$ denote the set of $\G$-measurable real-valued square integrable random variable, $\Mt{\G}$ denote the set of $\G$-measurable random probability measure $\mu$ on $\R$ with finite second moment, and $\H^2_\mb{G}([0,T];\R)$ denote the set of $\G_t$-progressively-measurable process $\beta = (\beta_{t})_{0 \leq t \leq T}$ such that
$$ \mbb{E}\l[ \int_{0}^{T} \beta^{2}_{t} dt \r] < \infty$$
We define similarly the space $\H_\mb{G}^2([s,t];\R)$ for any $0 \leq s  <  t \leq T$ and we will often omit the subscript and write $\H^2([0,T];\R)$ for $\H^2_\mb{F}([0,T];\R)$. We also let $\CP$ denote the space of continuous $\tF_t$-adapted stochastic flow of probability measure $\mu = (\mu_t)_{0 \leq  t \leq T}$ with value in $\Ptwo$ and define similarly $\MM{s}{t}$.

For a control $\alpha \in \Htwo$, let $(X^{\eps,\alpha}_t)_{0 \leq t \leq T}$ be the corresponding state process
\begin{equation}\label{state} 
X^{\eps,\alpha}_t = \xi_0 + \int_0^t \alpha_s ds + \sigma W_t + \eps \tW_t
\end{equation}
where $\xi_0 \in \Ltwo_{\F_0}$ is an initial state. Let $m^{\eps,\alpha}_t$ denote the law of $X^{\eps,\alpha}_t$ conditional on $\tF_t$, i.e.
$$ m^{\eps,\alpha}_t = \Law(X^{\eps,\alpha}_t | \tF_t) $$
It is easy to check that $m^{\eps,\alpha}=(m^{\eps,\alpha}_t)_{0 \leq t \leq T} \in \CP$ when $\alpha \in \Htwo$. Given any $m \in \CP $, we defined the expected cost corresponding to a control $\alpha$ as 
\begin{equation}\label{cost}
\Je(\alpha | m) \triangleq \mb{E}\l[ \int_0^T \frac{\alpha_t^2}{2}dt + g(X^{\eps,\alpha}_T,m_T)\r]
\end{equation}
where $g : \R \times \Ptwo \to \R $ is a terminal cost. 

\begin{remark} While we assume a quadratic running cost to simplify the notations, the result is expected to hold under a more general running cost satisfying similar assumptions that shall be imposed on the terminal cost function $g$, namely convexity and weak monotonicity. 
\end{remark} 

%\begin{assump}(Lipschitz in $x$)\label{lip_x}  For each $m \in \Ptwo$, $g_x(x,m)$ exists and is Lipschitz continuous in $x$ i.e. there exist a constant $K$ such that
%	$$ |g_x(x,m)-g_x(x',m)|  \leq K |x-x'| $$
%	for any $x,x' \in \R$
%\end{assump}
%\begin{assump}(Convexity) \label{convex} For any $x,x' \in \R, m \in \Ptwo$, 
%	\begin{equation}\label{convexity}	
%		(g_x(x,m)-g_x(x',m))(x-x') \geq 0
%	\end{equation}
%\end{assump}
%\begin{assump}(Lipschitz in $m$)\label{lip_m} $g_x$ is Lipschitz continuous in $m$ uniformly in $x$, i.e. there exist a constant $K$ such that
%	$$ |g_x(x,m)-g_x(x,m')|  \leq K \W_2(m,m') $$
%	for all $x \in \R, m,m' \in \Ptwo$, where $\W_2(m,m')$ is the second order Wasserstein metric defined by (\ref{wass}). This is equivalent to the following; for any $x \in \R, \xi,\xi'  \in \Ltwo_\F$
%	$$ |g_x(x,\Law(\xi))-g_x(x,\Law(\xi'))|  \leq K (\mb{E}|\xi-\xi'|^2)^{\frac{1}{2}} $$
%\end{assump}
%\begin{assump}(Monotonicity)\label{mon} For any $m,m' \in \Ptwo$ and any $\gamma \in \P_2(\R^2)$ with marginals $m,m'$ respectively,
%	$$ \int_{\R^2} \l[ (g_x(x,m)-g_x(y,m'))(x-y) \r] \gamma(dx,dy) \geq 0 $$
%	Equivalently, for any $\xi,\xi' \in \Ltwo_\F$, 
%	\begin{equation}\label{monotone}
%		\mb{E}[ g_x(\xi,\Law(\xi)) - g_x(\xi',\Law(\xi'))(\xi-\xi')] \geq 0
%	\end{equation}
%\end{assump}

Now we fix $m \in \CP$ and consider a stochastic control problem with the state process \rref{state} and the cost function $\Je(\alpha | m)$. We will refer to this problem as \textit{an individual control problem corresponding to $m$}. In the context of a stochastic differential game,  $m_t$ here represents the distribution of all the players in the game at time $t$. We would like to consider how each individual optimally chooses his/her control given such information. The MFG solution then represents a Nash equilibrium where every player is optimal given other players' decision. 

%
%Notice that in the $N$-player game, a control $\alpha$ of each player influence the empirical distribution at the order $\frac{1}{N}$. However, when we consider the limiting problem, i.e. take $N \to \infty$, an influence from any single player to the distribution disappears. As a result, one can solve the individual control problem with $m$ being fixed. This is the essence of mean field games which turns a complex interacting system to a single control problem with an equilibrium condition. For readers interested in the relation between the $N$-player game and the MFG, we refer to \cite{bensoussan2013,cardaliaguet2010,fischer2014,gueant2011,lacker2014general} for more details. 

The mean field game problem is defined as follows; Find a control $\ha \in \Htwo$ such that given $m^{\eps,\ha} = (m^{\eps,\ha}_t)_{0 \leq t \leq T}$, the optimal control for the state process (\ref{state}) with cost $\Je(\alpha | m^{\eps,\ha})$ defined by (\ref{cost}) is again $\ha$. In other words, $\ha \in \Htwo$ satisfies
$$ \Je(\ha | m^{\eps,\ha}) \leq \Je(\alpha | m^{\eps,\ha}) ,\quad \forall\alpha \in \Htwo.$$ 
It can be stated succinctly as 
\begin{equation}\label{mfg_c_2}\begin{cases}
    \alpha^* \in \arg\max_{\alpha \in \mc{A}} \mb{E} \l[ \int_0^T \frac{\alpha^2}{2}dt + g(X^\alpha_T,m_T)\r]\\
    dX^\alpha_t = \alpha_t dt + \sigma dW_t + \eps d\tW_t  \\
    m_t = \Law(X^{\alpha^*}_t | \tF_t), \quad \tF_t = \sigma(\tW_s; 0 \leq s \leq t) 
  \end{cases}
\end{equation}
We will often refer to this game as $\eps$-MFG to emphasize the level of the common noise term and call $\ha$ a \textit{solution to $\eps$-MFG problem with initial $\xi_0$}. Observe that the $\eps$-MFG described above can be viewed as a standard control problem with an additional fixed point feature.

\subsection{Existence and uniqueness of MFG with common noise}\label{eu_mfg}  Let us first state the assumptions on the cost function $g$
\begin{assump}(Lipschitz in $x$)\label{lip_x}  For each $x \in \R, m \in \Ptwo$, $\d{x}g(x,m)$ exists and is Lipschitz continuous in $x$. There exists a constant $K$ such that for any $x,x' \in \R$
	$$ |\d{x}g(x,m)-\d{x}g(x',m)|  \leq K |x-x'| $$
	
\end{assump}
\begin{assump}(Convexity) \label{convex} For any $x,x' \in \R, m \in \Ptwo$, 
	\begin{equation}\label{convexity}	
		(\d{x}g(x,m)-\d{x}g(x',m))(x-x') \geq 0
	\end{equation}
\end{assump}

%We can now state the SMP for the control problem for a given $m$ (Theorem \ref{smp}) and for $\eps$-MFG (Theorem \ref{smp_mfg_thm}). We refer the readers to \cite{ahuja2014} for the proof and Ch.3 in \cite{young1999} for a general reference on SMP.  
Under these assumptions, we can apply the SMP to the individual control problem for a given $m \in \CP$ resulting in the following system of FBSDE 
\begin{equation}\label{smp_fbsde}
	\begin{gathered}
 dX_{t} = -Y_{t} dt + \sigma dW_{t} + \eps \tdW_{t} \\
 dY_{t}  = Z_{t} dW_{t} + \tZ_{t}\tdW_{t} \\
 X_{0} = \xi, \quad Y_{T} = \d{x}g(X_T,m_T) 
	\end{gathered}
\end{equation}
Solving this FBSDE yields the optimal control $\ha^\eps_t = - Y_t $.
%\begin{thm}\label{smp}  Let $m \in \CP$, suppose \ref{lip_x}-\ref{convex} holds, then $\ha^{\eps} \in \Htwo$ is an optimal control to the individual control problem given $m$ if and only if there exists an $\F_t$-adapted solution $(\Xe_t,\Ye_t,\Ze_t,\tZe_t)_{0 \leq t \leq T}$ to the FBSDE
%\begin{equation}\label{smp_fbsde}
%	\begin{gathered}
% dX_{t} = -Y_{t} dt + \sigma dW_{t} + \eps \tdW_{t} \\
% dY_{t}  = Z_{t} dW_{t} + \tZ_{t}\tdW_{t} \\
% X_{0} = \xi, \quad Y_{T} = \d{x}g(X_T,m_T) 
%	\end{gathered}
%\end{equation}
%satisfying 
%$$\mb{E}\l[ \int_0^T (\Ye_t)^2 dt \r] < \infty, $$
%In that case $\ha^\eps_t = - \Ye_t $ and for any control $\beta \in \Htwo$, the following estimate holds
%\begin{equation}\label{smp_est}
% \Je(\beta | m)  - \Je(\ha^\eps | m) \geq \frac{1}{2}\mb{E}\l[ \int_0^T | \ha_t^\eps - \beta_t|^2 dt \r]  
% \end{equation}
%\end{thm}
The definition of $\eps$-MFG says that $(m_t)_{0 \leq t \leq T}$ must satisfy the following \textit{consistency} condition.
$$ m_t = m^{\eps,\ha^\eps}_t = \Law(X^{\eps,\ha^\eps}_t | \tF_t) $$
Adding this condition to \eqref{smp_fbsde}, we have the mean-field FBSDE corresponding to MFG with common noise ($\eps$-MFG)
\begin{equation}\label{smp_mfg}
	\begin{gathered}
 dX_{t} = -Y_{t} dt + \sigma dW_{t} + \eps \tdW_{t} \\
 dY_{t}  = Z_{t} dW_{t} + \tZ_{t}\tdW_{t} \\
 X_{0} = \xi, \quad Y_{T} = \d{x}g(X_T,\Law(X_T|\tF_T)) 
	\end{gathered}
\end{equation}
Note that the two system, \eqref{smp_fbsde} and \eqref{smp_mfg}, are different. The FBSDE \eqref{smp_fbsde} is a classical FBSDE with random coefficients from an \textit{exogeneous} $m$. The system \eqref{smp_mfg}, on the other hand, is a mean-field FBSDE where it depends on the law of the solution.

%\begin{thm}\label{smp_mfg_thm} Under \ref{lip_x}-\ref{convex}, there exist a solution to $\eps$-MFG if and only if the following FBSDE
%\begin{equation}\label{smp_mfg}
%	\begin{gathered}
% dX_{t} = -Y_{t} dt + \sigma dW_{t} + \eps \tdW_{t} \\
% dY_{t}  = Z_{t} dW_{t} + \tZ_{t}\tdW_{t} \\
% X_{0} = \xi, \quad Y_{T} = \d{x}g(X_T,\Law(X_T|\tF_T)) 
%	\end{gathered}
%\end{equation}
%is solvable. In that case, the $\eps$-MFG solution is given by $\ha^\eps_t = -\Ye_t$ where $(\Xe_t,\Ye_t,\Ze_t,\tZe_t)_{0 \leq t \leq T}$ denote a solution to the FBSDE \rref{smp_mfg}. 
%\end{thm}

We now discuss the solvability of this FBSDE under what we called a weak monotonicity condition on the cost function $g$. The result below is mainly taken from \cite{ahuja2014}, so we will state the main assumptions and results without proof and refer the reader to \cite{ahuja2014} and reference therein for more detail. We also discuss a slightly more general result, the existence and uniqueness of an FBSDE with monotone functionals, in Appendix \ref{fbsde_monotone_functionals} as we will be using such results in our subsequent analysis. We now state additional assumptions on $g$.

\begin{assump}(Lipschitz in $m$)\label{lip_m} $\d{x}g$ is Lipschitz continuous in $m$ uniformly in $x$, i.e. there exists a constant $K$ such that
	$$ |\d{x}g(x,m)-\d{x}g(x,m')|  \leq K \W_2(m,m') $$
	for all $x \in \R, m,m' \in \Ptwo$, where $\W_2(m,m')$ is the second order Wasserstein metric defined by (\ref{wass}). This is equivalent to the following; for any $x \in \R, \xi,\xi'  \in \Ltwo_\F$
	$$ |\d{x}g(x,\Law(\xi))-\d{x}g(x,\Law(\xi'))|  \leq K (\mb{E}|\xi-\xi'|^2)^{\frac{1}{2}} $$
\end{assump}
\begin{assump}(Weak monotonicity)\label{mon} For any $m,m' \in \Ptwo$ and any $\gamma \in \P_2(\R^2)$ with marginals $m,m'$ respectively,
	$$ \int_{\R^2} \l[ (\d{x}g(x,m)-\d{x}g(y,m'))(x-y) \r] \gamma(dx,dy) \geq 0 $$
	Equivalently, for any $\xi,\xi' \in \Ltwo_\F$, 
	\begin{equation}\label{monotone}
		\mb{E}[ \d{x}g(\xi,\Law(\xi)) - \d{x}g(\xi',\Law(\xi'))(\xi-\xi')] \geq 0
	\end{equation}
\end{assump}

%\begin{remark} We call the condition \rref{monotone} \textit{weak} monotone as it is weaker, when $g$ is convex in $x$, than the original monotonicity condition of Lasry and Lions \cite{cardaliaguet2010}. Examples of $g$ that satisfy \ref{lip_x}-\ref{mon} include quadratic cost functions such as  
%\begin{equation*}
% 	g(x,m) = \l( x - \int y dm(y) \r)^2, \qquad g(x,m) =  \int (x-y)^2 dm(y)
%\end{equation*}
% which occur frequently in applications (see \cite{carmona2013mean,gueant2011} for instance)\end{remark}

%It is straightforward to check that under  \ref{lip_x}-\ref{mon},  $G$ satisfies \rref{lipG} and \rref{monG}. Hence we can apply Theorem \ref{wellposed} to show that FBSDE \rref{smp_mfg} is uniquely solvable.

With the assumptions above, the existence and uniqueness of FBSDE \rref{smp_mfg} an the $\eps$-MFG follows. We refer to \cite{ahuja2014,ahuja2016forward} for more detail.

\begin{thm}[Well-posedness of $\eps$-MFG] Under \ref{lip_x}-\ref{mon}, there exist a unique solution $(X_t,Y_t,Z_t,\tZ_t)_{0\leq t \leq T}$ to FBSDE \rref{smp_mfg}. In particular, there exist a unique $\eps$-MFG solution for any initial $\xi_0 \in \Ltwo_{\F_0}$.
\end{thm}

%\begin{thm}[Wellposedness of McKean-Vlasov FBSDE] Under \ref{lip_x}-\ref{mon}, there exist a unique solution $(X_t,Y_t,Z_t,\tZ_t)_{0\leq t \leq T}$ to the FBSDE \rref{smp_mfg} satisfying
%\begin{equation}
%	\mb{E}\l[ \sup_{0 \leq t \leq T} [ X_t^2 + Y_t^2 ] + \int_0^T [Z_t^2 + \tZ_t^2] dt \r] \leq C\l( \mb{E}[\xi^2] + \d{x}g^2(0,\delta_0) + \sigma^2+\eps^2 \r) 
%\end{equation}
%where $\delta_a$ denote the dirac measure at $a$. Moreover, two solutions $(X^i_t,Y^i_t,Z^i_t,\tZ^i_t)_{0\leq t \leq T}, i=1,2$ to FBSDE \rref{smp_mfg} with initial $\xi_i$ satisfies the estimate
%\begin{equation}\label{estimate_diff}
%	\mb{E}\l[ \sup_{0 \leq t \leq T} \Ind_A \Delta X_t^2 + \sup_{s \leq t \leq T} \Ind_A \Delta Y_t^2  +  \int_s^T [\Ind_A\Delta Z_t^2 + \Ind_A\Delta \tZ_t^2] dt \r] \leq C \mb{E}[\Ind_A \xi^2]
%\end{equation}
%where $\Delta X_t = X^1_t - X^2_t$, $\Delta Y_t, \Delta Z_t, \Delta \tZ_t, \Delta \xi$ are defined similarly, $C$ is a constant depends only on $K$, and $A$ is an $\tF_0$-measurable set.
%\end{thm}
%
%Combining with Theorem \ref{smp_mfg_thm}, we have the following wellposedness result for $\eps$-MFG with common noise.
%
%\begin{corollary}[Wellposedness of $\eps$-MFG] Under \ref{lip_x}-\ref{mon}, there exist a unique $\eps$-MFG solution for any initial $\xi_0 \in \Ltwo_{\F_0}$.
%\end{corollary}

\subsection{Decoupling function, Markov property, and the master equation}\label{decoupling_markov}

%Generally, to solve an FBSDE, we need to find what is called a \textit{decoupling field}, a possibly random function describing the relation of the backward process $Y_t$ as a function of the forward process $X_t$. When the drift and the diffusion are deterministic, this function is deterministic and a solution to a quasilinear PDE. In that case, the FBSDE is said to be Markovian and we call the function a \textit{decoupling function}. See \cite{delarue2002existence,ma1994solving} for more detail on a decoupling function in a deterministic case and \cite{ma2011} for a more general case.
%
%While we should expect a similar result in the $\eps$-MFG problem when the running and terminal cost functions are deterministic, it is not obvious a priori that this Markovian property holds particularly in the case of common noise. For a fixed $m \in \CP$, we are in fact dealing with FBSDE with non-deterministic coefficients. Specifically, we have a path dependent terminal function $g(\cdot,m_T(\omega))$. However, as $g$ is still a deterministic function of $m$, it is reasonable to expect a Markovian property if we include an additional input, the current distribution of players, or in FBSDE context, the conditional distribution of the state process. \todo{not a main result} Our main result for this section is the following theorem which establishes the Markov property of the MFG with common noise.

A decoupling field of an FBSDE is a possibly random function which describes the relation of the backward process $Y_t$ as a function of the forward process $X_t$. When the coefficients in the FBSDE are deterministic, this function is also deterministic and satisfies a quasilinear PDE. In that case, the FBSDE is said to be Markovian and we call the function a \textit{decoupling function}. A decoupling function is useful not only for solving an FBSDE, the method called \textit{Four-step scheme} \cite{ma1994solving}, but also for understanding the connection between the SMP and HJB approach to stochastic control problems.

For $\eps$-MFG, the existence of a (deterministic) decoupling function is not obvious a priori particularly in the case of common noise since for a fixed $m \in \CP$,  we are in fact dealing with FBSDE with random coefficients. 
However, as the cost function are still a deterministic function of $m$, it is plausible to have such property if we include, as an additional input, the current distribution of players, or in FBSDE context, the conditional distribution of the state process. We state here such result for $\eps$-MFG which is proven in \cite{ahuja2016forward}. We also refer to \cite{delarue2014classical} for more detailed analysis of a decoupling function for 0-MFG.
 
\begin{thm}\label{ue_existence} Under \ref{lip_x}-\ref{mon}, there exist a deterministic function $\Ue: [0,T] \times \R \times \Ptwo \to \R$ such that
 \begin{equation}\label{decouple_smp} 
 	\Ye_t = \Ue(t,\Xe_t,\Law(\Xe_t|\tF_t))
\end{equation}
Moreover, $\Ue$ is uniformly Lipschitz; for all $t \in [0,T], x,x' \in \R, m,m' \in \Ptwo$,
$$| \Ue(t,x,m) - \Ue(t,x',m')| \leq C_{K,T} \l( | x-x'| + \W_2(m,m') \r)$$
where $C_{K,T}$ is a constant depending only on $K,T$.
%Moreover, $\Ue$ satisfies the estimates \todo{add t in estimate}
%\begin{enumerate}
%	\i $| \Ue(t,x,m) - \Ue(t,x',m')| \leq c_u \l( | x-x'| + \W_2(m,m') \r)$ \label{diffUe}
%	\i $\l( \Ue(t,x,m)-\Ue(t,x',m)\r)(x-x') \geq 0 $ \label{convexUe}
%	\i  $ \int_{\R^2} \l[ (\Ue(t,x,m)-\Ue(t,y,m'))(x-y) \r] \gamma(dx,dy) \geq 0 $ \label{monUe}
%\end{enumerate}
%for all $t \in [0,T], x,x' \in \R, m,m' \in \Ptwo$, where $\gamma \in \P_2(\R^2)$ with marginal law $m,m'$ respectively and $c_u$ depends only on $K$ and $T-t$. 
\end{thm}
As a consequence of the Markov property, the $\eps$-MFG solution is in the feedback form; that is,
 \begin{equation}\label{feedback optimal control} 
 \hae_t = -Y^\eps_t = -\Ue(t,\Xe_t,\Law(\Xe_t|\tF_t)) 
 \end{equation}
 
The decoupling function $\Ue$ in \eqref{decouple_smp} can be defined through a system of FBSDE as follows. For $(s,x,m) \in [0,T] \times \R \times \Ptwo $, we solve the following FBSDE\
\begin{equation}\label{mkfbsde_sub}
	\begin{gathered}
 dX_t = -Y_t dt + \sigma dW_{t} + \eps d\tW_t  \\
 dY_t  =Z_t dW_{t} + \tZ_t\tdW_{t} \\
 X_s = x, \quad Y_T = \d{x}g(X_T,m^{s,m}_T) 
	\end{gathered}
\end{equation}
where $(m^{s,m}_t)_{s \leq t \leq T}$ is the stochastic flow of $\eps$-MFG over $[s,T]$ with initial at $s$ = $m$. Denote the solution of \rref{mkfbsde_sub} by $(X^{s,x,m}_t,Y^{s,x,m}_t,Z^{s,x,m}_t,\tZ^{s,x,m}_t)_{s \leq t \leq T}$, then $Y^{s,x,m}_s$ is deterministic and we define $\Ue$ as 
$$ \Ue(s,x,m) \triangleq Y^{s,x,m}_s $$
We refer to \cite{ahuja2016forward} for more detail. Similar to the classical FBSDE, it is natural to ask if $\Ue$ is a solution to a certain PDE. It turns out that under certain condition, $\Ue:[0,T]\times \R \times \Ptwo \to \R$  can be characterized as a solution to the following \textit{master} equation
\begin{equation}\label{master_ue}
\begin{split}  \d{t}\Ue(t,x,m)-\Ue(t,x,m)\d{x}\Ue(t,x,m)+\frac{\sigma^{2}+\eps^2}{2}\dd{xx}\Ue(t,x,m)- \hE\l[\d{m}\Ue(t,x,m)(\hX)(\d{x}\Ue(t,\hX,m)) \r] \\ \quad + \frac{\sigma^2}{2}\dd{mm}\Ue(t,x,m)(\hX)[\zeta,\zeta] + \frac{\eps^2}{2}\dd{mm}\Ue(t,x,m)(\hX)[1,1]  +\eps^2\hE\l[\dd{xm}\Ue(t,x,m)(\hX)1 \r]= 0
\end{split}
\end{equation}
with terminal condition
$$ \Ue(T,x,m) = \d{x}g(x,m) $$
We refer to Proposition 4.1 in \cite{carmona2014master} for a related result. Here $\hX$ is a lifting random variable, i.e. $\Law(\hX)=m$, and $\zeta$ is a $\mc{N}(0,1)$-random variable independent of $\hX$ both of which are related to the notion of differentiation with respective to $m$ as described in Appendix \ref{derivative}.

This master equation is an infinite-dimensional HJB equation involving the derivative with respect to a probability measure. It was first introduced by Lasry and Lions and was discussed more extensively in \cite{bensoussan2014master,carmona2014master,delarue2014classical}\footnote{In particular, see equation (47) in \cite{carmona2014master}.}.  For the case $\eps=0$, namely MFG without common noise, Lasry and Lions propose a model in  \cite{cardaliaguet2010} through the following forward backward PDE
\begin{equation}\label{uomo}
\begin{aligned}
 	&\d{t}\uo =  \frac{(\d{x}\uo)^{2}}{2}-\frac{\sigma^{2}}{2}\dd{xx}\uo,  &\uo(T,x) = g(x,\mo_T) \\
	&\d{t}\mo = \d{x}(\d{x}\uo\mo) + \frac{\sg^{2}}{2}\dd{xx}\mo,  & \mo(0,x) = m_0(x) = \Law(\xi_0)
\end{aligned}
\end{equation}
where $m^0_t(\cdot) = m^0(t,\cdot)$. The first equation denotes the backward HJB equation for the value function of each players given the flow of distribution $(m^0_t)_{0 \leq t \leq T}$. The second equation is the forward Fokker-Planck equation describing the distribution of players' state given all the players adopt the strategy 
$$ \bal(t,x,m^0_t) = -\d{x}u^0(t,x) $$

Under sufficient regularity assumptions on $u^0$, it can be connected with $\U^0(t,x,m)$ via
\begin{equation}\label{uouox}
 \U^0(t,x,m^0_t) = \d{x}u^0(t,x)
\end{equation}

We would like to emphasize the relation \rref{uouox} as the terms $\U^0(t,x,m^0_t), \d{x}\U^0(t,x,m^0_t), \d{m}\U^0(t,x,m^0_t)$ are the main terms that will appear in our subsequent asymptotic analysis. The relation \rref{uouox} means that the first two terms can be found from the system of PDE \eqref{uomo} describing the $0$-MFG. The last term, which represents the sensitivity of the solution around the optimal path $(m^0_t)_{0 \leq t \leq T}$, is new and will be of crucial importance in the asymptotic analysis.  

%The Markov property and a decoupling function is not only interesting in its own right, but will also prove to be important in the asymptotic analysis. In Theorem \ref{}, we show that
%$$ \frac{\Ue(t,x,m) - \U^0(t,x,m)}{\eps} \to 0 \qquad\text{as }\eps \to 0 $$
%uniformly in $(t,x,m)$. As a result,
%$$ Y^\eps_t  \approx \U^0(t,X^\eps_t,\Law(X^\eps_t | \tF_t)) $$
%In other words, the $\eps$-first order approximation strategy is simply to follow the $0$-MFG solution along the trajectory $m^{\eps}$. We will return to this point again in the next section. In addition, understanding $\Ue$ will lead us to find a decoupling function for the linear FBSDE describing the first order expansion and enable us to solve it explicitly. See Section \ref{}.

\section{Asymptotic analysis}\label{asymptotic}

\subsection{Linear variational FBSDE} In the previous section, we have discussed that, under a linear-convex framework, finding a solution of the $\eps$-MFG is equivalent to solving the corresponding mean-field FBSDE \rref{smp_mfg}, and such system is in fact uniquely solvable under \ref{lip_x}-\ref{mon}. Let denote its solution by $(X^\eps_t, Y^\eps_t, Z^\eps_t, \tZ^\eps_t)_{0 \leq t \leq T}$, i.e. they satisfies
\begin{equation}\label{smp_mfg_2}
	\begin{gathered}
 d\Xe_{t} = -\Ye_{t} dt + \sigma dW_{t} + \eps \tdW_{t} \\
 d\Ye_{t}  = \Ze_{t} dW_{t} + \tZe_{t}\tdW_{t} \\
 \Xe_{0} = \xi, \quad \Ye_{T} = \d{x}g(\Xe_T,\Law(\Xe_T|\tF_T)) 
	\end{gathered}
\end{equation}

Solving this FBSDE yields the $\eps$-MFG solution by setting $\hal^\eps_t = -\Ye_t$. From the discussion in section \ref{decoupling_markov}, we see that solving the $0$-MFG  problem for $(X^0_t,Y^0_t, Z^0_t, \tZ^0_t)_{0 \leq t \leq T}$ requires us to find $\U^0$, which by \rref{uouox} is reduced to solving a system of PDEs. However, when adding common noise, we need to solve for $\Ue$, a solution to the master equation \rref{master_ue}. Instead of solving this infinite-dimensional equation, our goal here is to consider the approximation $(\Xe_t,\Ye_t, \Ze_t, \tZe_t)_{0 \leq t \leq T}$ around $(X^0_t,Y^0_t, Z^0_t, \tZ^0_t)_{0 \leq t \leq T}$ when the common noise is small. Equivalently, we would like to consider the limit as $\eps \to 0$ of
$$ \frac{\Xe_t-\Xo_t}{\eps}, \quad  \frac{\Ye_t-\Yo_t}{\eps}$$
First, we need an additional regularity assumption on $g$;

\begin{assump}\label{second} 
 $\d{x}g$ is differentiable in $(x,m)$ with Lipschitz continuous and bounded derivative. Denote the bound and Lipschitz constant by the same $K$. Specifically for $\dd{mx}g$, they satisfy, for all $x \in \R$, $m,m' \in \Ptwo$, and $\xi,\xi' \in \Ltwo_\F$ with law $m,m'$,
\be\label{lip_second}
\begin{gathered}
\mb{E}[ \dd{mx}g(x,m)(\xi)^2]^{\frac{1}{2}} \leq K \\
\mb{E}[ (\dd{mx}g(x,m)(\xi) - \dd{mx}g(x,m')(\xi'))^2]^{\frac{1}{2}} \leq K  \|\xi-\xi'\|_2
\end{gathered}
\ee
\end{assump}

\begin{remark} $\dd{xm}g$ involves a derivative with respect to a probability measure. We follow the framework introduced by Lasry and Lions in \cite{cardaliaguet2010} which is based on a \frechet derivative of a lifting function defined on a space of random variable. See Appendix \ref{derivative} for more detail.
\end{remark}

Let $\Delta X^\eps_t = \frac{\Xe_t-\Xo_t}{\eps}  $ and denote similarly $\dYe_t,\dZe_t,\dtZe_t$, then they satisfy
 \begin{equation}
 \label{fbsde_delta}
 \begin{gathered}
 	 d\dXe_t = - \dYe_t dt + d\tW_t \\
	 d\dYe_t = \dZe_t dW_t +\dtZe_t d\tW_t \\
 \dXe_0 = 0, \quad \dYe_T= \frac{\d{x}g(X^\eps_T, \Law(X^\eps_T | \tF_T)) -  \d{x}g(X^0_T, \Law(X^0_T | \tF_T))}{\eps}
\end{gathered}
\end{equation}
 Formally taking $\eps \to 0$, we get the following \textit{linear variational FBSDE}
 \begin{equation}
 \label{var}
 \begin{gathered}
 	 dU_t = - V_tdt + d\tW_t \\
	 dV_t =  Q_tdW_t + \tQ_td\tW_t \\
 U_0 = 0, \quad V_T= \dd{xx}g(X^0_T, m^0_T)U_T + \hat{\mb{E}}[ \dd{xm}g(X^{0}_{T},m^0_{T})(\hX^0_T)\hU_{T}]  
 %\mb{E}\l[ \dd{xm}\tg(X^0_T, X^0_T)U_T | \tF_T\r]
\end{gathered}
\end{equation}
 where 
 $$m^0_t = \Law(X^0_t | \tF_t) = \Law(X^0_t) $$
 and $\hX^0$ and $\hU$ are identical copies of $X^0$ and $U$ in the copied space $(\h{\Omega}^0 \times \t{\Omega},\h{\F} \otimes \tF,\h{\mb{P}} \otimes \tP)$ and $\hmbE[\cdot]$ is the expectation with respect to $\hw$ only. The copied space is used simply to distinguish a random variable used to represent a law in the \textit{lifting} functionals from the original random variable (see Appendix \ref{derivative} for more detail).
 We can write this term explicitly as
 \begin{equation}\label{expand_integral}
 \begin{aligned}
  \hat{\mb{E}}[ \dd{xm}g(X^{0}_{T},m^0_{T})(\hX^0_T)\hU_{T}]  
  &= \int_{\hat{\Omega}^0} \dd{xm}g(X^{0}_{T}(\omega^0),m^0_{T})(\hX^0_T(\hw))\hU_{T}(\hw,\tw) d\hat{\mb{P}}(\hw)   \\
  & = \int_{\Omega^0} \dd{xm}g(X^{0}_{T}(\omega^0),m^0_{T})(X^0_T(\hw))U_{T}(\hw,\tw) d\mb{P}^0(\hw)  
   \end{aligned}
   \end{equation}
where we suppress the $\tw$ in $X^0_T,\hX^0_T$ as they do not depend on it. One can see that the term $\hat{\mb{E}}[ \dd{xm}g(X^{0}_{T},m^0_{T})(\hX^0_T)\hU_{T}] $ is a mean field term that couples $\l\{ (\hX^0_T,U_T)(\hw,\cdot); \hw \in \Omega^0 \r\}$ together. Note that each path $ \hw \in \Omega^0$ corresponds a path of an individual player, so in other words, the mean field term gives the sensitivity to the first order change from all other players. Our first result in this section shows that this variation process is well-posed. 

%by using Theorem \ref{wellposed}.  

\begin{thm} Assume \ref{lip_x}-\ref{second} hold, there exists a unique adapted solution $(U_t,V_t,Q_t,\tQ_t)_{0 \leq t \leq T}$ to FBSDE \rref{var} satisfying
\begin{equation}\label{estimate_var}
	\mb{E}\l[ \sup_{0 \leq t \leq T} [ U_t^2 + V_t^2 ] + \int_0^T [Q_t^2 + \tQ_t^2] dt \r] \leq C_{K,T}
\end{equation}
where $C_{K,T}$ is a constant depending only on $K,T$.
\end{thm}
% \begin{proof} The FBSDE \rref{var} is a special case of \rref{fbsde_G} with $\sigma =0, \tsg = 1, X_0 = 0$, and 
% $$G(\xi) = \dd{xx}g(X^0_T, m^0_T)\xi + \h{\mb{E}}\l[ \dd{xm}g(X^{0}_{T},m^0_T)(\hX^0_T)\h{\xi} \r],$$ 
% where $\hX^0_T, \h{\xi}$ are identity copies of $X^0_T,\xi$ in the space  $(\h{\Omega} \times \t{\Omega},\h{\F} \otimes \tF,\h{\mb{P}} \otimes \tP)$. See Appendix \ref{derivative} for more detail on the derivative with respect to a probability measure. It is then easy to check that under \ref{lip_x}-\ref{second}, $G$ satisfies the assumption of Theorem \ref{wellposed}. 
% \end{proof}
  \begin{proof}  We define a functional $G: \Ltwo_\F \to \Ltwo_\F$ by 
 $$G(\xi) = \dd{xx}g(X^0_T, m^0_T)\xi + \h{\mb{E}}\l[ \dd{xm}g(X^{0}_{T},m^0_T)(\hX^0_T)\h{\xi} \r]$$
where $\hX^0_T, \h{\xi}$ are identical copies of $X^0_T,\xi$ in the copied space  $(\h{\Omega}^0 \times \t{\Omega},\h{\F} \otimes \tF,\h{\mb{P}} \otimes \tP)$ and $\hmbE[\cdot]$ is the expectation with respect to $\hw$ only. Then by \ref{lip_x}-\ref{second}, it follows that $G$ satisfies functional Lipschitz and monotone properties. That is, for any $\xi_1,\xi_2 \in \Ltwo_{\F}, A_T \in \tF_T$,
	\begin{gather*}
		%\item $\mb{E}[G_0^2] \leq c^2_G$
		 \mb{E}[\Ind_{A_T}(G(\xi_1)-G(\xi_2))^2] \leq K^2\mb{E}[\Ind_{A_T}(\xi_1-\xi_2)^2] \label{lipG} \\
		\mb{E}[\Ind_{A_T}(G(\xi_1)-G(\xi_2))(\xi_1-\xi_2)] \geq 0 \label{monG}
	\end{gather*}
The existence and uniqueness then follows from Theorem 1 in \cite{ahuja2016forward} (the statement is also provided in Appendix \ref{fbsde_monotone_functionals}). 
 \end{proof}

We are now ready to state our first main result which establishes the differentiability of $\eps$-MFG solution with respect to $\eps$.

\begin{thm}\label{diff} Assume \ref{lip_x}-\ref{second} hold, let $(\Xe_t,\Ye_t,\Ze_t,\tZe_t)_{0 \leq t \leq T}$ denote the solution to \rref{smp_mfg_2} and $(U_t,V_t,Q_t,\tQ_t)_{0 \leq t \leq T}$ denote the solution to \rref{var}, then there exist a constant $C_{K,T}$ depending only on $K,T$ such that
 \begin{equation}\label{converge_xy}
 \mb{E} \sup_{0\leq t \leq T} \l[\l(\frac{X^\eps_t-X^0_t}{\eps} - U_t\r)^2 + \l(\frac{Y^\eps_t-Y^0_t}{\eps} - V_t\r)^2 \r]  \leq C_{K,T}\eps^2
 \end{equation}
\end{thm}

\begin{proof} See Appendix \ref{proof_convergence}.
\end{proof}

We are able to obtain the convergence result above in a strong sense (in $\Ltwo$) mainly due to the monotone property of our setting. As seen in \cite{peng1999,hu1995,ahuja2016forward}, the monotone property of an FBSDE enables the proof for existence and uniqueness which relies on the standard SDE estimates and probabilistic tools. Similarly here, we are able to apply such tools to prove the convergence in $\Ltwo$.

%\begin{proof} \todo{here also forward reference}The proof is based on the estimate \rref{estimate} with the bounded second derivative of the cost function $g$. See Appendix \ref{proof_convergence} for the proof.
%\end{proof}

\subsection{Approximate Nash equilibrium for $\eps$-MFG}\label{approx_sec}

We have shown that
$$ \frac{\hae_t - \ha^0_t}{\eps} = \frac{-\Ye_t +Y^0_t}{\eps}\to  -V_t \qquad\text{as }\eps \to 0 $$
where the limit is in $\Htwo$. Using this result, we construct the first order approximate strategy by
\be
\beta^\eps_t \triangleq \ha^0_t - \eps V_t, \quad \forall t \in [0,T]
\ee
Being a game, an appropriate notion of approximation is required to see if $(\beta^\eps_t)_{0 \leq t \leq T}$ serves as a good approximation. In this case, it is reasonable to assume that each player adopts this approximate solution and analyze the gap between the expected cost under this set of strategies and the optimal cost. For an exact Nash equilibrium, this gap is precisely zero by definition as every player is optimal given the other players' strategy. This notion of approximate optimality is called $\delta$-Nash equilibrium.\footnote{It is conventionally called $\eps$-Nash equilibrium. We use the parameter $\delta$ here to avoid confusion with the parameter $\eps$ denoting the level of common noise} In a finite-player game, it is defined as follows. 

\begin{definition} Under the same notations as defined in section \ref{mfg_commonnoise}, for the $N$-player game, a set of admissible strategies $(\alpha^i_t)_{0 \leq t \leq T, 1 \leq i \leq N}$ is called a $\delta$-Nash equilibrium if for each $i=1,2,\dots,N$,
$$ \J^i\l(\alpha^i | (\alpha^j)_{j \neq i} \r) \leq  \J^i\l(\beta| (\alpha^j)_{j \neq i} \r) + \delta  $$
for all $\beta =(\beta_t)_{0 \leq t \leq T} \in \Htwo$ where $\J^i(\cdot)$ denote the cost function of player $i$.
\end{definition}  

To go from a finite-player symmetric game to its limit, we formally take $N \to \infty$, assume that each player adopts the same strategy, and use a single player as a representative player. As a result, we can define an approximate Nash equilibrium similarly for MFG as follows;

\begin{definition}\label{approx_nash_def} Under the same notations as defined in section \ref{mfg_commonnoise}, an admissible strategy $\alpha = (\alpha_t)_{0 \leq t \leq T} \in \Htwo$ is called a $\delta$-Nash equilibrium for $\eps$-MFG problem if
$$ \Je(\alpha | m^{\alpha}) \leq \Je(\beta | m^{\alpha}) + \delta $$
for all $\beta = (\beta_t)_{0 \leq t \leq T}$ where $\Je(\cdot)$ denotes the cost function and $m^\alpha_t$ denotes the conditional law of $X^\alpha_t$ with $(X^\alpha_t)_{0 \leq t \leq T} $ being the state process corresponding to $\alpha$.
\end{definition}  

\begin{remark} By definition, an $\eps$-MFG solution is a $0$-Nash equilibrium for an $\eps$-MFG problem.
\end{remark}

The notion of an approximate Nash equilibrium is important in the theory of stochastic games with infinite horizon where for many problems, there is no exact Nash equilibrium while there exists a $\delta$-Nash equilibrium for any $\delta > 0$. It is also a widely used notion in \textit{algorithmic game theory} where the main interest is in finding a polynomial time algorithm that yields an approximate Nash equilibrium solution when finding an exact Nash equilibrium is computationally expensive. 

In MFG, this notion is used mainly in the study of the relation between an MFG and a symmetric $N$-player stochastic differential game. Recall that the motivation for considering an MFG model is in its application for finding a good approximate strategy for an $N$-player game when $N$ is large. In \cite{carmona2013probabilistic}, Carmona and Delarue show that under a linear-convex MFG model without common noise, the $0$-MFG strategy is $\eps_N$-Nash equilibrium for the corresponding $N$-player game with $\eps_N \sim O(N^{-1/(d+4)})$ where $d$ is the dimension of the underlying Euclidean space. See also \cite{carmona2013weak, huang2007, li2011} for other similar results. The converse, which asks whether the Nash equilibrium from $N$-player game converges to a corresponding MFG solution, is also of interest and is more challenging. For interested readers, we refer to \cite{fischer2014} and reference therein for results in this direction all of which are for MFG models without common noise.
  
In this paper, we are only concerned with the model at the limit with a continuum of players. We are particularly interested in an approximate solution for $\eps$-MFG using the information from $0$-MFG solution. Our main result for this section is the following theorem

\begin{thm}\label{approx_nash} Assume\ref{lip_x}-\ref{second} hold. For $\eps > 0$, let $\hae=(\hae_t)_{0 \leq t \leq T}$ denote the solution to the $\eps$-MFG and $(U_t,V_t,Q_t,\tQ_t)_{0 \leq t \leq T}$ denote the solution to the linear variation FBSDE \rref{var}. Define a first order approximate strategy $\beta^\eps = (\beta^\eps_t)_{0 \leq t \leq T}$ by
\be
\beta^\eps_t \triangleq \ha^0_t - \eps V_t
\ee
Then $\beta^\eps$ is an $\eps^2$-Nash equilibrium for $\eps$-MFG.
\end{thm}

\begin{proof} See Appendix \ref{proof_approx_nash}.
\end{proof}

\subsection{Decoupling function of the linear variation FBSDE}  Despite being linear, the FBSDE \rref{var} is not trivial to solve due to the mean field term $\hE\l[ \dd{xm}g(X^0_T,m_T^0)(\hX^0_T)\hU_T \r]$. Nonetheless, we proceed in the similar way as solving a classical FBSDE by attempting to find a decoupling function describing $V_t$ as a function of $U_t$. Recall that we have a decoupling function $\Ue$ satisfying the relation
 $$ Y_t^{\eps} = \Ue(t,X^{\eps}_t,\Law(X^{\eps}_t|\tF_t)) $$
Therefore, we can write
\begin{equation}\label{ue_decompose}
\begin{aligned}
V_t 	&= \lim_{\eps \to 0} \frac{\Ye_t-Y^0_t}{\eps}  \\
	&= \lim_{\eps \to 0} \frac{\Ue(t,\Xe_t,\Law(\Xe_t| \tF_t))- \U^0(t,X^0_t,\Law(X^0_t))}{\eps}  \\
 	&= \lim_{\eps \to 0} \frac{\Ue(t,\Xe_t,\Law(\Xe_t| \tF_t))-\U^0(t,\Xe_t,\Law(\Xe_t| \tF_t))+\U^0(t,\Xe_t,\Law(\Xe_t| \tF_t))-\U^0(t,X^0_t,\Law(X^0_t))}{\eps} \\
	&= \lim_{\eps \to 0} \frac{\Ue(t,\Xe_t,\Law(\Xe_t| \tF_t))-\U^0(t,\Xe_t,\Law(\Xe_t| \tF_t))}{\eps}+ \lim_{\eps \to 0}\frac{\U^0(t,\Xe_t,\Law(\Xe_t| \tF_t))-\U^0(t,X^0_t,\Law(X^0_t))}{\eps} 
\end{aligned}
\end{equation}
where the limit is in $\Ltwo_{\F}$. The proposition below shows that the first part is in fact zero.

%and $\hX^0_t,\hU_t$ is a copy of $X^0_t,U_t$ in $(\h{\Omega} \times \t{\Omega},\h{\F} \otimes \tF,\h{\mb{P}} \otimes \tP)$.
%The following proposition shows that the first term is in fact zero and we have a decoupling field of FBSDE \rref{var}.

\begin{prop}\label{meanzero_smp} Assume \ref{lip_x}-\ref{mon} holds. Let $\Ue$ denote the decoupling function of FBSDE \rref{smp_mfg} as defined in \rref{decouple_smp}, then the following holds;
\begin{equation}\label{ueuo} 
	\lim_{\eps \to 0} \frac{\Ue(t,x,m) - \U^0(t,x,m)}{\eps} = 0 
\end{equation}
uniformly in $(t,x,m) \in [0,T] \times \R \times \Ptwo$.
\end{prop}

\begin{proof} See Appendix \ref{proof_ueuo}
\end{proof}

The result above implies that, at the first order, the decoupling function for $\eps$-MFG and $0$-MFG is the same. Combining with the recent result by Chassagneux et al.\cite{delarue2014classical} which proves the existence of a classical solution $\U^0$ of equation \rref{master_ue} with $\eps=0$, we have the decoupling \textit{functional} for FBSDE \rref{var}. To apply such result, an extra regularity assumption for $g$ is needed

\begin{assump}\label{third} For all $m \in \Ptwo$, the map $(x,z) \mapsto\dd{mx}g(x,m)(z),$ is continuously differentiable and satisfies for all  $x,x',\alpha \in \R$,$m,m' \in \Ptwo$, and $\xi,\xi' \in \Ltwo_\F$ with law $m,m'$,
\be\label{lip_second_m}
\begin{gathered}
\mb{E}\l[\l( \d{z}\dd{xm}g(x,m)(\xi) - \d{z}\dd{xm}g(x',m')(\xi')\r)^2\r]^{\frac{1}{2}} \leq K \l( |x-x'| + \|\xi-\xi'\|_2  \r)
\end{gathered}
\ee
\end{assump}

\begin{remark} The map $(x,z) \mapsto\dd{xm}g(x,m)(z)$ is related the notion of derivative with respect to a probability measure as described in Appendix \ref{derivative}. For instance, if $f(x,m) =  (x-\int_{\R} y dm(y))^2$, then the lifting functional is $\tilde{f}(x,\xi) = (x- \mb{E}[\xi])^2$, and the \frechet derivative is given by  $D\tilde{f}(x,\xi)= 2(x-\mb{E}[\xi]) = \mb{E}[ 2(x-\xi)]$. Thus, $\d{m}f(m)(z) = 2(x-z)$. 
\end{remark} 

The following theorem gives the decoupling functional for the linear variational process \rref{var}.

\begin{thm}\label{decouple_uv} Assume \ref{lip_x}-\ref{third} holds. Let $(U_t,V_t,Q_t,\tQ_t)_{0 \leq t \leq T}$ denote the solution to \rref{var}, $\Uo$ denote the decoupling function for $0$-MFG defined in section \ref{decoupling_markov}, then 
\begin{equation}\label{vt}
V_t = \d{x}\U^0(t,X_t^0, m^0_t)U_t + \hE[\d{m}\U^0(t,X_t^0, m^0_t)(\hX^0_t)\hU_t]  
\end{equation}
\end{thm}
\begin{proof} From \eqref{ue_decompose} and Proposition \ref{meanzero_smp}, we have
$$ V_t = \lim_{\eps \to 0}\frac{\U^0(t,\Xe_t,\Law(\Xe_t| \tF_t))-\U^0(t,X^0_t,\Law(X^0_t))}{\eps} $$
From Theorem 5.5 in \cite{delarue2014classical}, we have that $\d{x}\Uo, \d{m}\Uo$ exist, and they are bounded by Theorem \ref{ue_existence}. The result then follows from Theorem \ref{diff} above.
\end{proof}

From Theorem \ref{decouple_uv} above, we have decoupled the FBSDE \rref{var} and get the following forward mean-field SDE
\begin{equation}\label{dU}
 dU_t = - \l[ \d{x}\U^0(t,X_t^0, m^0_t)U_t + \hE[\d{m}\U^0(t,X_t^0, m^0_t)(\hX^0_t)\hU_t] \r] dt + d\tW_t, \quad U_0 = 0 
 \end{equation}

%Assume that $\U^0$ is differentiable, then using the result above, we have a decoupling function
%\begin{equation}\label{vt}
%V_t = \d{x}\U^0(t,X_t^0, m^0_t)U_t + \hE[\d{m}\U^0(t,X_t^0, m^0_t)(\hX^0_t)\hU_t]  
%\end{equation}
%Plugging this back in and we have decoupled the FBSDE \rref{var} and get the following forward SDE
%\begin{equation}\label{dU}
% dU_t = - \l[ \d{x}\U^0(t,X_t^0, m^0_t)U_t + \hE[\d{m}\U^0(t,X_t^0, m^0_t)(\hX^0_t)\hU_t] \r] dt + d\tW_t, \quad U_0 = 0 
% \end{equation}

Proposition \ref{meanzero_smp} has a simple yet interesting implication. It says that to approximate the $\eps$-MFG solution at the first order, we simply need to use the $0$-MFG solution applying along the trajectory $(t,\Xe_t,\Law(\Xe_t|\tF_t))$, i.e.
$$ \alpha^\eps_t = -\Ye_t = -\Ue(t,\Xe_t,\Law(\Xe_t|\tF_t)) \approx -\U^0(t,\Xe_t,\Law(\Xe_t|\tF_t))$$
However, we do not usually know $\U^0(t,x,m)$ for all $(t,x,m)$ but only $\U^0(t,x,m^0_t)$ where $m^0_t = \Law(X^0_t)$ corresponds to the $0$-MFG solution. The full information of $\U^0$ at every point $(t,x,m)$ will require solving the master equation \rref{master_ue} which is infinite-dimensional problem and is non-trivial to do so. On the other hands, $\U^0(t,x,m^0_t)$ is simply a gradient of a solution of the forward-backward PDE \eqref{uomo} of Lasry and Lions. So unless we know the function $\U^0(t,x,m)$, this process means that to get our optimal control at every time $t$, we need to resolve $0$-MFG problem over $[t,T]$ with initial $m_t = \Law(X^\eps_t| \tF_t)$ which is computationally expensive. As a result, we need to approximate $\U^0$ at the current state $(t,\Xe_t,\Law(\Xe_t|\tF_t))$ by $\U^0(t,X^0_t,m^0_t)$. In fact, it is not necessary to approximate at $(t,X^0_t,m^0_t)$ if we can observe $X^\eps_t$. In other words, making use of \rref{ueuo}, we can expand around $(t,\Xe_t,m^0_t)$ instead and get a slightly simpler approximation of $\ha^\eps_t$ as follows;
\begin{equation}\label{vt2}
\begin{aligned}
  \ha^\eps_t = -\Ye_t &=  -\Ue(t,\Xe_t,\Law(\Xe_t|\tF_t)) \\
  					&= -\U^0(t,\Xe_t,\Law(\Xe_t|\tF_t)) + o(\eps)  \\
					&= -\U^0(t,\Xe_t,m^0_t) + \eps \hE[\d{m}\U^0(t,\Xe_t, m^0_t)(\hX^0_t)\hU_t] + o(\eps)
 \end{aligned}
 \end{equation}
From both \rref{vt} and \rref{vt2}, we see that the derivative with respect to the $m$-argument of $\U^0(t,x,m^0_t)$ along the direction $\hU_t$ is essential in our asymptotic analysis.

%From both \rref{vt} and \rref{vt2}, we see that whichever approximation we choose, the crucial term is $\hE[\d{m}\U^0(t,X_t^0, m^0_t)(\hX^0_t)\hU_t]$.

% We would like to point out that the SDE \rref{dU} is non-standard since it involves the mean field term in the form of the conditional law of $U_t$. Recall that $(\hU_t)_{0 \leq t \leq T}$ is a copy of $(U_t)_{0 \leq t \leq T}$ sharing the common noise space. We let $\gamma_t,\beta_{s,t},\eta_t \in \Ltwo_{\F_t}, \eta: [0,T] \times \R \times \t{\Omega} \to \R$ be defined as
%$$ \gamma_t  \triangleq \d{x}\U^0(t,X_t^0, m^0_t),\quad  \beta_{s,t} \triangleq e^{- \int_s^t  \gamma_r dr }, \quad  \eta(t,x,\tw) \triangleq \hE[\d{m}\U^0(t,x,m_t^0)(\hX^0_t)\hU_t],\quad \eta_t \triangleq \eta(t,X^0_t,\tw) $$
% then from \rref{vt} and \rref{dU};
% \begin{equation}\label{uv_explicit}
% \begin{gathered}
%U_t = - \int_0^t \beta_{s,t}  \eta_s ds - \int_0^t \beta_{s,t} d\tW_s  \\
%V_t = \gamma_tU_t + \eta_t 
%\end{gathered}
%\end{equation}
%\todo{revise this, mention something we might not do/discuss}
% is a solution to FBSDE \rref{var}, hence the unique solution. The term $\beta_{s,t}, \gamma_t$ can be derived from a $0$-MFG solution straightforwardly, by solving \rref{uomo} and using \rref{uouox} for instance. We are now left to analyze a more non-trivial term, the random function $\eta(t,x,\tw)$. 

%\subsection{Centered Gaussian process}

Having established the first order approximation of the $\eps$-MFG solution in the form of a solution to a linear variational FBSDE, we now proceed to analyze properties of the solution $(U_t,V_t)_{0 \leq t \leq T}$. While the FBSDE \rref{var} describing them seems complicated as it involves both the individual noise and common noise, this is simply a nature of SMP approach as it describes the control in the open-loop form (a function of path) instead of the closed-loop feedback form (a function of state). However, if we only analyze the effect of the perturbation by the common noise, or equivalently if we look at the distribution of $(U_t(\omega^0,\cdot),V_t(\omega^0,\cdot))_{0 \leq t \leq T}$ for a fixed $\omega^0 \in \Omega^0$, then they are simply a pair of centered Gaussian process. 

\begin{thm}\label{meanzero} Let $(U_t,V_t,Q_t,\tQ_t)_{0 \leq t \leq T}$ denote the solution to the FBSDE \rref{var}, then $(U_t(\omega^0,\cdot), V_t(\omega^0,\cdot))_{0 \leq t \leq T}$ is a pair of Gaussian process in $(\t{\Omega},\{\tF_t\}_{0 \leq t \leq T}, \tP)$ with mean zero $\mb{P}^0$-a.s.
\end{thm}  

%\begin{proof} The Gaussian property follows from the fact that \rref{dU} is linear with respect to the common noise and $(V_t)_{0 \leq t \leq T}$ is linear with respect to $(U_t)_{0 \leq t \leq T}$. The mean zero is proved by taking expectation and show that the resulting system has a unique solution which is zero. See Appendix \ref{proof_meanzero} for the proof.
%\end{proof}

\begin{proof}
Recall that $\F^0$ denote the $\sigma$-algebra in the first component of the product sample space which is independent of the common noise filtraiton. Let $A: [0,T] \times \Ltwo_{\F^0} \to \Ltwo_{\F^0}$ denote the linear operator 
$$ A(t,\xi) = \d{x}\U^0(t,X_t^0, m^0_t)\xi + \hE[\d{m}\U^0(t,X_t^0, m^0_t)(\hX^0_t)\h{\xi}]  $$
Then we can view \eqref{dU} as a stochastic evolution equation in $\Ltwo_{\F^0}$
$$ dU_t = A(t,U_t)dt + d\tW^H_t, \quad U_0=0$$
where  $(\tW^H_t)_{0 \leq t \leq T}$ denote the natural lifting of the common noise $(\tW_t)_{0 \leq t \leq T}$ to a Gaussian process in $\Ltwo_\F$ along the constant direction. That is, if $e_1=1$ denote the constant random variable, then for any $\xi \in \Ltwo_{\F}$, 
$$ \la \tW^H_t, \xi \ra = \la e_1, \xi \ra \tW_t = \mb{E}[\xi] \tW_t $$
By the uniform Lipschitz property of $\Uo$ (see Theorem \ref{ue_existence}), we have that $\| \d{x}\U^0(t,X^0_t,m^0_t) \|_\infty$, $\hE[\d{m}\U^0(t,X^0_t,m^0_t)(\hX^0_t)^2]^{\frac{1}{2}}$ are bounded. Thus, $A$ is a bounded linear functional and hence induces a strongly continuous semigroup $(S(s,t))_{0 \leq s \leq t \leq T}:  \Ltwo_{\F^0} \to \Ltwo_{\F^0}$. Therefore, by the variational of constant formula (see Theorem 5.4 in \cite{da2014stochastic}) and $U_0=0$, we have 
$$ U_t = S(0,t)U_0 + \int_0^t S(s,t)e_1 d\tW_s =  \int_0^t S(s,t)e_1 d\tW_s $$
Thus, $U_t$ is a mean-zero Gaussian process with respect to the common noise and the result follows since $V_t$ is a linear function of $U_t$.
\end{proof}

\section{Connection to Dynamic Programming Principle}\label{dpp2}
%Mean field games was first proposed by Lasry and Lions through the Dynamic Programming Principle resulting in a coupled forward-backward PDE of HJB equation and Fokker-Plank equation. We would like to discuss here the connection between the SMP approach we use here and the PDE approach of Lasry and Lions. 

In this section, we discuss a connection between the SMP approach and the DPP approach and present asymptotic results from the DPP approach. Please note that the results here are largely formal and are intended to give a connection to a more familiar DPP approach. 

\subsection{FBSPDE for $\eps$-MFG} We follow the same method used to derive the system of PDEs \rref{uomo} for the $0$-MFG. We will attempt to write the forward-backward equations where the forward one describes the evolution of the equilibrium distribution through the Fokker-Planck equation and the backward one describes the HJB equation of the value function.

%%% Copied from section 2.4. Start.

%%Move back to section 2.4
%We first note that the decoupling function $\Ue$ in \eqref{decouple_smp} can be defined through a system of FBSDE as follows. For $(s,x,m) \in [0,T] \times \R \times \Ptwo $, we solve the following FBSDE\
%\begin{equation}\label{mkfbsde_sub}
%	\begin{gathered}
% dX_t = -Y_t dt + \sigma dW_{t} + \eps d\tW_t  \\
% dY_t  =Z_t dW_{t} + \tZ_t\tdW_{t} \\
% X_s = x, \quad Y_T = \d{x}g(X_T,m^{s,m}_T) 
%	\end{gathered}
%\end{equation}
%where $(m^{s,m}_t)_{s \leq t \leq T}$ is the flow of $\eps$-MFG over $[s,T]$ with initial at $s$ = $m$. Denote the solution of \rref{mkfbsde_sub} by $(X^{s,x,m}_t,Y^{s,x,m}_t,Z^{s,x,m}_t,\tZ^{s,x,m}_t)_{s \leq t \leq T}$, then $Y^{s,x,m}_s$ is deterministic and we define $\Ue$ as 
%$$ \Ue(s,x,m) = Y^{s,x,m}_s $$

Recall from equation \eqref{feedback optimal control} that the optimal control of $\eps$-MFG in feedback form is given by $-\Ue(s,x,m)$. We can then define the value function $\Ve: [0,T]\times \R \times \Ptwo \to \R$ as 
\begin{align*}
 \Ve(t,x,m) &=  \inf_{(\alpha_{s})_{t \leq s \leq T} \in \HtwotT}  \mathbb{E}\l[ \int_{t}^{T} \alpha_{s}^{2} ds + g(\Xe_{T},\me_{T}) \Big| \Xe_{t} = x, \me_{t} = m\r] \\
 	&= \mathbb{E}\l[ \int_{t}^{T} \Ue(s,X^\eps_s,m^\eps_s)^{2} ds + g(\Xe_{T},\me_{T}) \Big| \Xe_{t} = x, \me_{t} = m\r]
\end{align*}
where $\me_t = \Law(\Xe_t | \tF_t)$. The value function above represents the minimum expected cost from $t$ to $T$ given the state of the game at time $t$ (the current state of a player and the current distribution of all players). Suppose that $\Ve$ is sufficiently regular, then by Dynamic Programming Principle, it can be shown to satisfy

\begin{equation}\label{master}
\begin{split}  \d{t}\Ve(t,x,m)-\frac{(\d{x}\Ve(t,x,m))^{2}}{2}+\frac{\sigma^{2}+\eps^2}{2}\dd{xx}\Ve(t,x,m)- \hE\l[\d{m}\Ve(t,x,m)(\hX)(\d{x}\Ve(t,\hX,m)) \r] \\ \quad + \frac{\sigma^2}{2}\dd{mm}\Ve(t,x,m)(\hX)[\zeta,\zeta] + \frac{\eps^2}{2}\dd{mm}\Ve(t,x,m)(\hX)[1,1]  +\eps^2\hE\l[\dd{xm}\Ve(t,x,m)(\hX)1 \r]= 0
\end{split}
\end{equation}
with terminal condition
$$ \Ve(T,x,m) = g(x,m) $$
where $\hX$ is a lifting random variable, i.e. $\Law(\hX)=m$, and $\zeta$ is a $\mc{N}(0,1)$-random variable independent of $\hX$. The derivative with respect to the $m$-argument is based on the framework proposed by Lasry and Lions in \cite{cardaliaguet2010} as introduced in the previous section. We refer the reader to Appendix \ref{derivative} and reference therein for more detail.

The connection between the SMP and the HJB approach for a general stochastic control problem is well understood. That is, the backward process is the gradient of the value function, at least when the value function is sufficiently regular. A more general statement can be said in term of sub/super gradient and a viscosity solution (see Ch.5 in \cite{young1999} for instance). Similarly for $\eps$-MFG, we have
\begin{equation}\label{ueve}
\Ue(t,x,m) = \d{x}\Ve(t,x,m)
\end{equation}
This relation can be proved in a similar way by taking \ito's lemma on $\d{x}\Ve(t,\Xe_t,\Law(\Xe_t | \tF_t))$ and use \rref{master} to show that it satisfies \rref{smp_mfg}. See Theorem 6.4.7 in \cite{pham2009} for the proof when there is no argument $m$ and section 6 in \cite{carmona2014master} for a generalized \ito's lemma with a probability measure argument.

%%% Copied from section 2.4. End.
From \eqref{feedback optimal control}, \eqref{ueve}, we have that the $\eps$-MFG Nash equilibrium strategy is 
\begin{equation}\label{control}
 \ha^\eps_t = -\Ye_t = -\Ue(t,\Xe_t,\Law(\Xe_t|\tF_t)) = -\d{x}\Ve(t,\Xe_t,\Law(\Xe_t|\tF_t))
 \end{equation}
Let $\me_t = m^{\hae}_t$ denote the corresponding conditional law, then it satisfies the following stochastic Fokker-Plank equation 
\begin{equation}\label{FP1}
	 d\me(t,x) = \l(\d{x}(\d{x}\Ve(t,x,\me_t)\me_t) + \frac{\sg^{2}+\eps^{2}}{2}\dd{xx}\me_t\r)dt -\eps \d{x}\me_t\;d\tW_{s}
\end{equation}
Now we define a random value function along $(\me_t)_{0 \leq t \leq T}$ by letting
$$ \ue(t,x,\tw) = \Ve(t,x,\me_t(\tw)) $$
Using the master equation \rref{master}, it follows that 
\begin{align*}
 d\ue(t,x) &= \Bigg(  \d{t}\Ve(t,x,\me_t) + \frac{\sigma^2}{2}\dd{mm}\Ve(t,x,\me_t)(\hX)[\zeta,\zeta] + \hE\l[\d{m}\Ve(t,x,\me_t)(\hX)(\d{x}\Ve(t,\hX,\me_t)) \r]  \\
 &\qquad \qquad +\frac{\eps^2}{2}\dd{mm}\Ve(t,x,\me_t)(\hX)[1,1  \Bigg) dt - \eps \hE\l[\d{m}\Ve(t,x,\me_t)(\hX)1 \r] d\tW_{t} \\
 & = \l( \frac{(\d{x}\Ve(t,x,\me_t))^{2}}{2}-\frac{\sigma^{2}}{2}\dd{xx}\Ve(t,x,\me_t) - \frac{\eps^{2}}{2}\Big( \dd{xx}\Ve(t,x,\me_t)- 2 \hE\l[\dd{xm}\Ve(t,x,\me_t)(\hX)1 \r]\Big)\r)dt \\
 &\qquad\qquad - \eps  \hE\l[\d{m}\Ve(t,x,\me_t)(\hX)1 \r] d\tW_{t} \\
 & = \l( \frac{(\d{x}\ue(t,x))^{2}}{2}-\frac{\sigma^{2}}{2}\dd{xx}\ue(t,x) - \frac{\eps^{2}}{2}\Big( \dd{xx}\ue(t,x)- 2 \hE\l[\dd{xm}\Ve(t,x,\me_t)(\hX)1 \r]\Big)\r)dt  \\
 &\qquad\qquad- \eps  \hE\l[\dd{xm}\Ve(t,x,\me_t)(\hX)1 \r] d\tW_{t} \\
 &= \l( \frac{(\d{x}\ue(t,x))^{2}}{2}-\frac{\sigma^{2}}{2}\dd{xx}\ue(t,x) - \frac{\eps^{2}}{2}\Big( \dd{xx}\ue(t,x)- 2 \d{x}w(t,x)\Big)\r)dt - \eps  \ve(t,x) d\tW_{t}
 \end{align*}
where
$$ \ve(t,x) \triangleq  \hE\l[\d{m}\Ve(t,x,\me_t)(\hX)1 \r] $$
Combining with (\ref{control}) and (\ref{FP1}) , we have arrived at a system of FBSPDE
\begin{equation}\label{ueme}
\begin{aligned}
	d\me(t,x) &= \l(\d{x}(\d{x}\ue\me) + \frac{\sg^{2}+\eps^2}{2}\dd{xx}\me+ \frac{\eps^{2}}{2}\dd{xx}\me\r)dt -\eps  \d{x}\me\;d\tW_{t} \\
 	d\ue(t,x) & = \l( \frac{(\d{x}\ue)^{2}}{2}-\frac{\sigma^{2}}{2}\dd{xx}\ue - \frac{\eps^{2}}{2}\Big( \dd{xx}\ue- 2\d{x}\ve\Big)\r)dt - \eps \ve d\tW_{t}
\end{aligned}
\end{equation}
with boundary conditions
$$ \me(0,x) = m_0(x) = \Law(\xi_0) , \qquad \ue(T,x) = g(x,\me_T) $$

Similarly to the equation \eqref{master_ue}, we have a verification theorem for \rref{ueme} which states that if we have a sufficiently regular solution $(\ue,\me,\ve)$ to the FBSPDE \rref{ueme} above, then the $\eps$-MFG solution is given by
$$ \ha_t^\eps = -\d{x}\ue(t,\Xe_t), $$ 
We refer the reader to section 4.2 in \cite{bensoussan2014master} for such result. The tuple $(\ue,\me,\ve)$ then gives, respectively, the value function, distribution of the optimal state process, and the sensitivity of the valuation with respect to a spatial shift of the distribution process. Consequently, despite the fact that the derivation of \rref{ueme} above requires the regularity of a solution of the master equation, it actually contains the same information as the master equation. One represents the value function at time $t$ as a function of $(x,m)$ while the other represents the value function at time $t$ as a function of $(x,\tw)$ where $\tw$ is a common Brownian motion path. Note also that when $\eps=0$, as expected, we get back the system of PDEs \rref{uomo} of Lasry and Lions.

\subsection{Asymptotic analysis} We now consider the case when $\eps$ is small and the approximation of ($\ue,\me$) around ($\uo,\mo$). Particularly, we want to find a pair of random functions $(\dmo,\duo)$ from the following expansion
\begin{align*}
\me(t,x,\tw) &= \mo(t,x) + \eps \dmo(t,x,\tw) + o(\eps) \\
\ue(t,x,\tw) &= \uo(t,x) + \eps \duo(t,x,\tw) + o(\eps)
\end{align*}
Let us proceed formally. We write
$$\du(t,x) = \frac{\ue(t,x) - \uo(t,x)}{\eps}, \quad \dm(t,x) = \frac{\me(t,x) - \mo(t,x)}{\eps} $$
Then from the dynamic of ($\ue,\me$) and  ($\uo,\mo$) in \rref{ueme}, it follows that
\begin{align*}
	d\dm(t,x)	&= \l[ \frac{\sg^{2}}{2}\dd{xx}\dm+\frac{(\me \d{x}\ue - \mo\d{x}\uo )_{x}}{\eps} + O(\eps) \r] dt -  \d{x}\me d\tW_t \\
				&=  \l[ \frac{\sg^{2}}{2}\dd{xx}\dm+\d{x}(\dm\d{x}\uo+\mo\d{x}\du) + O(\eps) \r] dt -  \d{x}\me d\tW_t
\end{align*}
and 
\begin{align*}
	d\du(t,x) 	&=  \l[ \frac{(\d{x}\ue)^{2}-(\d{x}\uo)^{2}}{2\eps} - \frac{\sg^{2}}{2}\dd{xx}\du + O(\eps) \r] dt -  \hE\l[\d{m}\Ve(t,x,\me_t)(\hX)1 \r] d\tW_t \\
				&= \l[ \d{x}\du\d{x}\uo - \frac{\sg^{2}}{2}\dd{xx}\du + O(\eps) \r] dt -   \hE\l[\d{m}\Ve(t,x,\me_t)(\hX)1 \r] d\tW_t 
\end{align*}
with boundary conditions
$$ \dm(0,x) = m_0(x), \qquad \du(T,x) = \frac{g(x,\me_T)-g(x,\mo_T)}{\eps} $$
Formally taking the limit as $\eps \to 0$, we obtain the system of SPDEs describing the $\eps$-correction terms,
\be\label{fbspde1}
\begin{aligned}
	d\dmo(t,x)	&=  \l[ \frac{\sg^{2}}{2}\dd{xx}\dmo+\d{x}(\dmo\d{x}\uo+\mo\d{x}\duo) \r] dt -  \d{x}\mo d\tW_t  \\
	d\duo(t,x) 	&= \l[ \d{x}\duo\d{x}\uo - \frac{\sg^{2}}{2}\dd{xx}\duo  \r] dt -   \hE\l[\d{m}\Vo(t,x,\me_t)(\hX)1 \r] d\tW_t  
\end{aligned}
\ee
with boundary conditions given by
\be\label{boundary}
 \dmo(0,x) = 0,\qquad \duo(T,x) = \la \d{m}g(x,\mo_T)(\cdot),\dmo_T \ra
 \ee
where $\d{m}g$ denote the derivative with respect to the probability measure and $\la \cdot, \cdot \ra$ denote the inner product in $\Ltwo$. In this case, we are using a different notion of derivative, namely the G\^{a}uteax directional derivative, as it is more appropriate for this approach. We refrain from discussing in detail here and rather refer the reader to \cite{bensoussan2014master} for more detail.
%Gauteax derivative of 

Normally, in the BSPDE or BSDE setting, the diffusion part of the backward process is not specified and is part of a solution to ensure the adaptedness of a solution. In other words, the FBSPDE above should be written as
\be\label{fbspde2}
\begin{aligned}
	d\dmo(t,x)	&=  \l[ \frac{\sg^{2}}{2}\dd{xx}\dmo+\d{x}(\dmo\d{x}\uo+\mo\d{x}\duo) \r] dt -  \d{x}\mo d\tW_t  \\
	d\duo(t,x) 	&= \l[ \d{x}\duo\d{x}\uo - \frac{\sg^{2}}{2}\d{x}\duo  \r] dt - \dv d\tW_t  
\end{aligned}
\ee
and its solution would be a tuple of adapted-process $(\dmo,\duo,\dv)$. The FBSPDE \rref{fbspde2} is essentially the first-order correction term of $\eps$-MFG solution from the DPP approach. 

To see the connection between the two approaches, we express the stochastic function $(\duo,\dmo)$ in terms of $(U_t,V_t)$ and the derivatives of $\Uo$, the main objects from the SMP approach. The relation between $\duo,\dv$ and those from the SMP are clear from the relation $\d{x}\Ve(t,x,m) = \Ue(t,x,m)$. That is,
$$ \d{x}\duo(t,x) =  \lim_{\eps \to 0} \frac{\d{x}\Ve(t,x,\me_t)-\d{x}\Vo(t,x,\mo_t)}{\eps}  =  \hE\l[\d{m}\Uo(t,x,\mo_t)(\hX)(\hU_t) \r] $$
%$$ w(t,x) =  \hE\l[\d{m}\Uo(t,x,\me_t)(\hX)1 \r] =  \hE\l[\dd{xm}\Vo(t,x,\me_t)(\hX)1 \r] = \d{x}\dv(t,x) $$
%and
%$$ \eta(t,x) = \lim_{\eps \to 0} \frac{\Ue(t,x,\me_t)-\Uo(t,x,\mo_t)}{\eps} = \lim_{\eps \to 0} \frac{\d{x}\Ve(t,x,\me_t)-\d{x}\Vo(t,x,\mo_t)}{\eps} = \d{x}\duo(t,x) $$
The relation between $\dmo$ and those from the SMP is less straightforward mainly due to a difference in the notion of derivatives with respect to the law. For a test function $\phi$, we write
\begin{align*}
 \la \phi, \dmo(t,\cdot) \ra &= \lim_{\eps\to0} \frac{\la \phi, \me_t \ra - \la \phi, \mo_t \ra }{\eps} \\
 &=  \lim_{\eps\to 0} \frac{ \mb{E}[\phi(\Xe_t | \tF_t)] - \mb{E}[\phi(\Xo_t)] }{\eps} \\
 &= \mb{E}[ \d{x}\phi(\Xo_t)U_t | \tF_t]  \\
  &= \int \d{x}\phi(x)u m^{\Xo,U}_t(x,u) dxdu \\
  &= \la \d{x}\phi, \int_\R u m^{\Xo,U}_t(\cdot,u) du \ra \\
 &= \la \phi, - \d{x}\int_\R um^{\Xo,U}_t(\cdot,u) du \ra
 \end{align*} 
 where $m_t^{\Xo,U}$ denote the joint law of $\Xo_t,U_t$ conditional on $\tF_t$. That is, 
 $$ \dmo(t,x) =  - \d{x}\int_\R um^{\Xo,U}_t(x,u) du. $$

\appendix

%cite result instead
%\section{Proof of Theorem \ref{wellposed}}\label{proof_wellposed}
%\input{tex/proof_wellposed}

\section{FBSDE with monotone functionals}\label{fbsde_monotone_functionals}

 We state here an existence and uniqueness result for an FBSDE with monotone functionals. The result in this section is mainly from \cite{ahuja2016forward} with simpler setting. We consider an FBSDE of the form
 \begin{equation}
 \label{fbsde_G}
 \begin{gathered}
 	 dX_t = - Y_tdt + \sigma dW_t + \eps d\tW_t \\
	 dY_t =  Z_tdW_t + \tZ_td\tW_t \\
X_s = \xi, \quad Y_T = G(X_T)
 %\mb{E}\l[ \tg_{xm}(X^0_T, X^0_T)U_T | \tF_T\r]
\end{gathered}
\end{equation}
where $G : \Ltwo_{\F} \to \Ltwo_{\F}$ is a ``functional". This type of equation covers most, if not all, of the FBSDEs encountered in this paper. For instance, given $m \in \CP$, then 
$$ G(X) = \d{x}g(X,m_T)$$
corresponds to the FBSDE \rref{smp_fbsde} arising from an individual control problem for a given $m$. Another case and perhaps a more important one is when
$$ G(X) = \d{x}g(X,\Law(X|\tF_T)) $$
as it corresponds to the mean-field FBSDE \rref{smp_mfg} arising from the $\eps$-MFG. The first theorem concerns the wellposedness of FBSDE \rref{fbsde_G} when $G$ is Lipschitz and monotone under conditional expectation.

\begin{thm}\label{wellposed} Let $\xi \in \Ltwo_{\F}$ and $G: \Ltwo_{\F}\to  \Ltwo_{\F}$ be a functional satisfying the following Lipschitz and monotonicity conditions; there exist a constant $K$ such that for any $\xi_1,\xi_2 \in \Ltwo_{\F}, A_T \in \tF_T$,
	\begin{gather}
		%\item $\mb{E}[G_0^2] \leq c^2_G$
		 \mb{E}[\Ind_{A_T}(G(\xi_1)-G(\xi_2))^2] \leq K^2\mb{E}[\Ind_{A_T}(\xi_1-\xi_2)^2] \label{lipG} \\
		\mb{E}[\Ind_{A_T}(G(\xi_1)-G(\xi_2))(\xi_1-\xi_2)] \geq 0 \label{monG}
	\end{gather}
Then there exist a unique adapted solution $(X_t,Y_t,Z_t,\tZ_t)_{s \leq t \leq T}$ to FBSDE \rref{fbsde_G}
satisfying the estimate: for any $A \in \F_s$
\begin{equation}\label{estimate}
	\mb{E}\l[ \Ind_A  \l( \sup_{s \leq t \leq T} X_t^2 + \sup_{s \leq t \leq T} Y_t^2  + \int_s^T [Z_t^2 + \tZ_t^2] dt \r) \r] \leq C_K\l( \mb{E}[\Ind_A(\xi^2 + G(0)^2 )] + (\sigma^2+\eps^2)T \r) 
\end{equation}
where $C_K$ is a constant depends only on $K$.
\end{thm}

The second result gives the estimate of the solution.

\begin{thm}\label{main_estimate} Under the same assumption as Theorem \ref{wellposed}, let $\xi_1,\xi_2 \in \Ltwo_{\F_s}$ and $(X^i_t,Y^i_t,Z^i_t,\tZ^i_t)_{s \leq t \leq T}$, $i=1,2$ denote the solution of FBSDE \rref{fbsde_G} with initial $\xi_i$, then for any $\tF_s$-measurable set $A$,
\begin{equation}\label{estimate_diff}
	\mb{E}\l[ \sup_{s \leq t \leq T} \Ind_A \Delta X_t^2 + \sup_{s \leq t \leq T} \Ind_A \Delta Y_t^2  +  \int_s^T [\Ind_A\Delta Z_t^2 + \Ind_A\Delta \tZ_t^2] dt \r] \leq C_{K,T} \mb{E}[\Ind_A \Delta\xi^2]
\end{equation}
where $\Delta X_t = X^1_t - X^2_t$ and $\Delta Y_t, \Delta Z_t, \Delta \tZ_t, \Delta \xi$ are defined similarly, and $C_{K,T}$ is a constant depending only on $K,T$. 
\end{thm}

The proofs of Theorem \ref{wellposed} and \ref{main_estimate} make use of the monotonicity condition in a similar way as done in a classical FBSDE \cite{peng1999}. The proofs can be found in \cite{ahuja2016forward}  under more general functionals (see Theorem 3).

\section{Proof of Theorem \ref{diff}}\label{proof_convergence}
\begin{proof}
Let $\Delta X^{\eps}_t = \frac{X^\eps_t-X^0_t}{\eps}$ and $\delta^{X,\eps}_{t} =  \Delta X^{\eps}_{t} - U_{t}$ and define similarly $\Delta Y^{\eps}_t, \Delta Z^{\eps}_t, \Delta \tZ^{\eps}_t, \delta^{Y,\eps}_{t}, \delta^{Z,\eps}_t, \delta^{\tZ,\eps}_t$, then $(\delta^{X,\eps}_{t},\delta^{Y,\eps}_{t},\delta^{Z,\eps}_t, \delta^{\tZ,\eps}_t)_{0 \leq t \leq T}$ satisfies
\begin{equation}
\begin{gathered} 
	d \delta^{X,\eps}_{t} =  - \delta^{Y,\eps}_{t}dt, \\
	d \delta^{Y,\eps}_{t} = \delta^{Z,\eps}_tdW_t + \delta^{\tZ,\eps}_t d\tW_t, \\
	 \delta^{X,\eps}_{0} = 0, \quad \delta^{Y,\eps}_{T} = \frac{\d{x}g(X^{\eps}_{T},\Law(X^{\eps}_{T}|\tF_T)) -\d{x}g(X^{0}_{T},m^0_T)}{\eps} - \dd{xx}g(X^{0}_{T},m^0_T)U_{T} -\h{\mb{E}} \l[ \dd{xm}g(X^{0}_{T},m^{0}_{T})(\hX^0_T)\hU_{T} \r]
\end{gathered}
\end{equation}
Let
$$ X^{\ld,\eps}_t := X^{0}_t + \ld (X^{\eps}_t - X^{0}_t), \quad 0 \leq \ld \leq 1 $$
Note that
\begin{align*} 
	 & \frac{\d{x}g(X^{\eps}_{T},\Law(X^{\eps}_{T}|\tF_T)) -\d{x}g(X^{0}_{T},\Laww{X^{0}_{T}})}{\eps} - \dd{xx}g(X^{0}_{T},\Laww{X^{0}_{T}})U_{T} -\h{\mb{E}} \l[ \dd{xm}g(X^{0}_{T},m^{0}_{T})(\hX^0_T)\hU_{T} \r]\\
 &=  \int_{0}^{1} \l( \dd{xx}g(X^{\ld,\eps}_{T},\Law(X^{\ld,\eps}_{T}|\tF_T))\Delta X^{\eps}_{t} + \h{\mb{E}}\l[ \dd{xm}g(X^{\ld,\eps}_{T},\Law(X^{\ld,\eps}_{T}|\tF_T))(\hX^{\ld,\eps}_T)\Delta \hX^{\eps}_{t} \r]  \r) d\ld \\
 &\quad - \dd{xx}g(X^{0}_{T},m^0_T)U_{T} -\h{\mb{E}} \l[ \dd{xm}g(X^{0}_{T},m^{0}_{T})(\hX^0_T)\hU_{T} \r] \\
 &= \l[ \int_{0}^{1} \dd{xx}g(X^{\ld,\eps}_{T},\Law(X^{\ld,\eps}_{T}|\tF_T))d\ld \r] \delta^{X,\eps}_{T} + \int_{0}^{1} \h{\mb{E}}\l[ \dd{xm}g(X^{\ld,\eps}_{T},\Law(X^{\ld,\eps}_{T}|\tF_T))(\hX^{\ld,\eps}_T)\h{\delta}^{X,\eps}_{T} \r] d\ld \\
 &\quad + \l[  \int_{0}^{1} \dd{xx}g(X^{\ld,\eps}_{T},\Law(X^{\ld,\eps}_{T}|\tF_t))d\ld - \dd{xx}g(X^{0}_{T},\Law(X^{0}_{T})) \r] U_{T} \\
 &\quad +  \int_{0}^{1} \h{\mb{E}} \l[ \l( \dd{xm}g(X^{\ld,\eps}_{T},\Law(X^{\ld,\eps}_{T}|\tF_T))(\hX^{\ld,\eps}_{T}) -\dd{xm}g(X^{0}_{T},\Law(X^{0}_{T}))(\hX^0_T) \r) \hU_{T} \r] d\ld  \\
 & =: I^{\eps}_{1}( \delta^{X,\eps}_{T}) + I^{\eps}_{2},
 \end{align*}
where $I^{\eps}_1: \Ltwo_{\F_T} \to \Ltwo_{\F_T}$ is a linear functional and $I^{\eps}_2 \in \Ltwo_{\F_T}$. Because $\dd{xx}g, \dd{xm}g$ are bounded and $U_T,\hU_T$ are bounded in $\L^2$, it follows from estimate (\ref{estimate}) that 
$$ \mb{E}[\sup_{0 \leq t \leq T} (\delta^{X,\eps}_{t})^{2} ] \leq C_{K,T}, $$
where $C_{K,T}$ is a constant independent of $\eps$. As a result, we get
$$\mb{E}[ \sup_{\substack{0 \leq t \leq T \\ 0 \leq \ld \leq 1}} (X^{\ld,\eps}_{t}-X^{0}_{t})^{2} ] \leq C_{K,T}\eps^2. $$
Thus, there exist a constant $\tC_{K,T}$ depending only on $K,T$ such that
$$ \mb{E}[(I^{\eps}_{2})^2] \leq \tC_{K,T}\eps^2. $$ 
Then by the estimate (\ref{estimate}) again, we get \rref{converge_xy} as desired.
\end{proof}

\section{Proof of Theorem \ref{approx_nash}}\label{proof_approx_nash}
\begin{proof} As we are interested in the asymptotic limit as $\eps \to 0$, we assume without loss of generality that $|\eps| < 1$. Let $\hae = (\hae_t)_{0 \leq t \leq T}$ denote the $\eps$-MFG solution, $\beta^\eps = (\beta^\eps_t)_{0 \leq t \leq T}$ be the approximate strategy defined by
$$ \beta^\eps_t = \ha^0_t -\eps V_t $$
where $V=(V_t)_{0 \leq t \leq T}$ is the backward process of the linear variational FBSDE \rref{var}. For notational convenience, we will write $\Je(\alpha | \beta)$ to denote $\Je(\alpha | m^{\beta})$ for any $\alpha,\beta \in \Htwo$ (see section \ref{setup} for the definition of $\Je(\alpha | m^\beta)$ and $m^\beta$).

For any control $\alpha,\beta^{(1)},\beta^{(2)} \in \Htwo$, let $X^{\alpha},X^{\beta^{(1)}},X^{\beta^{(2)}}$ denote the corresponding state processes, then we have 
$$ \mb{E}[ (X^{\beta^{(1)}}_T -X^{\beta^{(2)}}_T)^2] \leq C_T \int_0^T | \beta^{(1)}_t-\beta^{(2)}_t|^2 dt $$  
Thus, combining with Lipschitz assumption on $g$, it follows that
\begin{equation}\label{approx_est1}
\begin{aligned}
| \Je(\alpha | \beta^{(1)}) - \Je(\alpha | \beta^{(2)}) | &\leq \mb{E}\l[g(X^{\alpha}_T,\Law(X^{\beta^{(1)}}_T|\tF_T) - g(X^{\alpha}_T,\Law(X^{\beta^{(2)}}_T|\tF_T))\r]  \\
&\leq K (\mb{E}[ (X^{\beta^{(1)}}_T -X^{\beta^{(2)}}_T)^2])^{\frac{1}{2}}  \\
&\leq C_{K,T} \l(\int_0^T | \beta^{(1)}_t-\beta^{(2)}_t|^2 dt \r)^{\frac{1}{2}}
\end{aligned}
\ee
Also, since $\hae$ is the $\eps$-MFG solution, it is an optimal control of the individual control given $m^{\hae}=(m^{\hae}_t)_{0 \leq t \leq T}$. Thus, we have the following estimate (see Theorem 2.2 in \cite{carmona2013probabilistic})
\be\label{approx_est2}
\Je(\hae | \hae) + C \int_0^T | \hae_t - \alpha_t|^2 dt \leq \Je(\alpha | \hae)
\ee
for any $\alpha \in \Htwo$. Lastly, from the definition of $\eps$-MFG strategy, we have
\be\label{approx_est3}
\Je(\hae | \hae) \leq \Je (\alpha | \hae)
\ee
for any $\alpha \in \Htwo$. Combining \rref{approx_est1},\rref{approx_est2}, and \rref{approx_est3} yields
\begin{align*}
\Je(\beta^\eps | \beta^\eps )-\Je(\alpha | \beta^\eps) &\leq \Je(\beta^\eps | \beta^\eps )-\Je(\hae | \hae)+ \Je (\alpha | \hae)-\Je(\alpha | \beta^\eps)  \\
&= \Je(\beta^\eps | \beta^\eps )-\Je(\beta^\eps | \hae) + \Je(\beta^\eps|\hae)-\Je(\hae | \hae) \\
&\qquad+ \Je (\alpha | \hae)-\Je(\alpha | \beta^\eps) \\
&\leq C\l[ \l( \int_0^T | \hae_t - \beta^\eps_t|^2 dt \r)^{\frac{1}{2}} + \int_0^T | \hae_t - \beta^\eps_t|^2 dt \r]
\end{align*}
Using estimate \rref{converge_xy} in Theorem \ref{diff}, it follows that
$$ \l( \int_0^T | \hae_t - \beta^\eps_t|^2 dt \r)^{\frac{1}{2}} =|\eps| \l\| \frac{\hae-\ha^0}{\eps} - V_t  \r\|_{\Htwo} \leq C_{K,T}\eps^2 $$
Thus, there exist a constant $\t{C}_{K,T}$ such that 
$$ \Je(\beta^\eps | \beta^\eps )-\Je(\alpha | \beta^\eps)  \leq \t{C}_{K,T}\eps^2 $$
for any $\alpha \in \Htwo$, $|\eps| \leq 1$ as desired.
\end{proof}

%\section{Proof of Theorem \ref{meanzero}}\label{proof_meanzero}
%\input{tex/proof_meanzero}

\section{Proof of Proposition \ref{meanzero_smp}}\label{proof_ueuo}
\begin{proof}
Fix $(s,x,m) \in [0,T]\times \R \times \Ptwo$, let $(X^\eps_t,Y^\eps_t,Z^\eps_t,\tZ^\eps_t)_{s\leq t \leq T}$ denote the solution to FBSDE \rref{mkfbsde_sub} over $[s,T]$ with initial $X_s = x$ and $(X^0_t,Y^0_t,Z^0_t,\tZ^0_t)_{s\leq t \leq T}$ denote the solution to FBSDE \rref{mkfbsde_sub} with $\eps=0$ and $\xi = x$. Recall that by the definition of $\Ue$ (see section \ref{decoupling_markov}), we have
$$ Y^{\eps}_s = \Ue(s,x,m), \quad Y^0_s = \Uo(s,x,m) $$ 
By the same argument as in Theorem \ref{diff}, we get
 \begin{equation}\label{converge_yx}
 \mb{E} \sup_{s\leq t \leq T} \l[\l(\frac{X^\eps_t-X^0_t}{\eps} - \bU_t\r)^2 + \l(\frac{Y^\eps_t-Y^0_t}{\eps} - \bV_t\r)^2 \r] \leq C_{K,T}\eps^2
 \end{equation}
where $C_{K,T}$ is a constant depending only on $K,T$ and not on $\eps,s,x,m$, and $(\bU_t,\bV_t, \bQ_t, \tbQ_t)_{s \leq t\leq T}$ satisfies
 \begin{equation}
 \label{var_x}
 \begin{gathered}
 	 d\bU_t = - \bV_tdt + d\tW_t \\
	 d\bV_t =  \bQ_tdW_t + \tbQ_t d\tW_t \\
 \bU_s = 0, \bV_T= \dd{xx}g(X^0_T, m^0_T)\bU_T + \hat{\mb{E}}[ \dd{xm}g(X^{0}_{T},m^0_{T})(\hX^0_T)\hU_{T}] , 
 %\mb{E}\l[ \t\dd{xm}g(X^0_T, X^0_T)U_T | \tF_T\r]
\end{gathered}
\end{equation}
where $(\hU_t,\hV_t)_{s \leq t\leq T}$ denote a copy in $(\h{\Omega},\h{\F},\h{\mb{P}})$ of the solution $(U_t,V_t)_{s \leq t \leq T}$ of FBSDE \rref{var} over $[s,T]$, and $\hat{\mb{E}}[\cdot]$ denote the expectation with respect to $\h{\mb{P}}$.

%, and $(X^{0}_t)_{s \leq t \leq T}$ denote the forward process of FBSDE \rref{mkfbsde_sub} with $\eps=0$ and $\xi = x$; that is,
%
%\todo{Double check/rewrite proof. Also, why we still have $\tZ^{0}_t\tdW_{t}$ in the dynamic of $Y^{0}_t$? }
%\begin{equation}\label{mkfbsde_sub_0}
%	\begin{gathered}
% dX^{0}_t = -Y^{0}_t dt + \sigma dW_{t}  \\
% dY^{0}_t  =Z^{0}_t dW_{t} + \tZ^{0}_t\tdW_{t} \\
% X^{0}_s = x, \quad Y^{0}_T = \d{x}g(X^{0}_T,m^{0,s,m}_T) 
%	\end{gathered}
%\end{equation}

Now let $\tE[\cdot]$ denote the expectation with respect to the common Brownian motion $(\tW_t)_{0 \leq t \leq T}$, i.e. with respect to $\tP$. Then since $X^0_t$ is independent of the common Brownian motion and also from Theorem \ref{meanzero} which says that $\hU_{T}$ has mean zero, it follows that $(\tE[\bU_t], \tE[\bV_t])_{s \leq t \leq T}$ satisfies the FBSDE
 \begin{equation}
 \begin{gathered}
 	 d\tE[\bU_t] = - \tE[\bV_t]dt, \\
	 d\tE[\bV_t ]=  C_tdW_t \\
 \tE[\bU_s] = 0, \quad \tE[\bV_T]= \dd{xx}g(X^0_T, m^0_T)\tE[\bU_T]
 %\mb{E}\l[ \t\dd{xm}g(X^0_T, X^0_T)U_T | \tF_T\r]
\end{gathered}
\end{equation}
Note that zero is a solution to this FBSDE and by uniqueness, it must be the only solution (see Theorem 2.2 in \cite{peng1999} or Theorem \ref{wellposed} in Appendix \ref{fbsde_monotone_functionals}). Therefore, 
\begin{equation}\label{euiszero}
 \tE[\bU_t] = \tE[\bV_t] = 0, \qquad \text{for }s \leq t \leq T 
 \end{equation}
 Combining with \rref{converge_yx} and the fact that $\Ue(s,x,m), \U^0(s,x,m)$ are deterministic, we get
 $$ \lim_{\eps\to 0} \sup_{(s,x,m) \in [0,T]\times \R \times \Ptwo } \l|\frac{\Ue(s,x,m)-\U^0(s,x,m)}{\eps} \r|^2 = 0 $$
as desired.  
 \end{proof}

\section{Derivative with respect to a probability measure}\label{derivative}
From the set up of MFG problem, we see that the distribution of player evolves stochastically and, as a result,  some notion of optimization, hence differentiation, over a probability measure is necessary. In this section, we discuss a notion of derivative for a function with a probability measure as its argument. 

A notion of derivative of a function on the space of probability measure was first defined using a geometric approach. See \cite{ambrosio2005,villani2009} for extensive treatments on the subject in this direction. In this work, however, it is more convenient to use an alternative approach which is more probabilistic in nature. To the best of our knowledge, this method was first introduced by P.L.Lions in his lecture at \college, which can be found in Ch.6 of \cite{cardaliaguet2010}. Since then, many works particularly those involve MFG with common noise have employed this notion of derivative. While we will only discuss the results that are relevant to our work here, we refer the interested readers to \cite{carmona2015} or more recently \cite{delarue2014classical} for more detail on this framework.

The idea is based on \textit{lifting up} a function on a space of probability measure to a function on a space of random variable. When the space of probability measure we are working on is $\Ptwo$, this method is extremely useful because it allows us to work on a Hilbert space of square integrable random variable instead of a metric space $\Ptwo$. Consequently, we are able to use a notion of Fr\'echet derivative in Hilbert space to help define a derivative on a space of probability measure.

Let $F$ be a continuous function from $\Ptwo$ to $\R$. Let $(\hat{\Omega},\hat{\F},\hat{\mb{P}})$ be an arbitrary probability measure space such that $\hat{\Omega}$ is a Polish space, $\hat{\mb{P}}$ is an atomless measure. We call a function $\t{F}:\Ltwo(\hat{\Omega};\R) \to \R$ an extension of $F$ if
$$ \t{F}(X) = F(\Law(X)) $$
where $\Law(X)$ denote the law of $X$. Note that $\t{F}$ is a map from Hilbert space $\Ltwo(\hat{\Omega})$ to $\R$, so we can use the notion of Fr\'echet derivative of $\t{F}$ to define a derivative of $F$. 

\begin{definition} Let $m_0 \in \Ptwo$, $F$ is differentiable at $m_{0}$ if there exist an extension $\t{F}$ and $X_0 \in \Ltwo(\hat{\Omega})$ such that $\Law(X_0)=m_0$ and $\t{F}$ is Fr\'echet differentiable at $X_{0}$.

\end{definition}

Suppose $\t{F}$ is Fr\'echet differentiable at $X_{0}$, then by Reiz Representation Theorem, there exist $DF(X_{0}) \in  \Ltwo(\hat{\Omega})$ such that
$$ \lim_{Y\to 0} \frac{| \t{F}(X_{0}+Y)-\t{F}(X_{0})-\hat{\mb{E}}[D\t{F}(X_{0})Y] |}{\|Y\|_2} = 0 $$
where $\| \cdot \|_2$ denote the $\Ltwo(\h{\Omega}) $ norm.

It can be shown (see Theorem 6.2 of \cite{cardaliaguet2010}) that the law of $D\t{F}(X_{0})$ is independent of the choice of lifting $X_{0}$.  In addition, it can also be shown (see Theorem 6.5 of \cite{cardaliaguet2010}) that there exist $h \in \L^2_{m_0}(\R,\R)$ uniquely defined $m_0$-almost everywhere such that for any lifting choice $X_0$, $D\t{F}(X_0) = h(X_0)$. Thus, it is natural to define this function $h \in \Ltwo_{m_0}(\R,\R)$ to be the derivative of $F$ with respect to $m$ at $m=m_0$. 

%In addition, if $D\t{F} \in \Ltwo(\h{\Omega})$ is Lipschitz with constant $c_F$, then we can select a version of $h \in \Ltwo  Lipschitz continuous with the same constant.

\begin{definition} The derivative of $F$ with respect to $m$ at $m=m_0$, denoted by $\d{m}F(m_0)$, is a measurable function from $\R \to \R$ such that
$$\lim_{Y \to 0} \frac{ F(\Law(X_0+Y)) - F(\Law(X_0)) - \hat{\mb{E}}[\d{m}F(m_0)(X_0)Y] }{\| Y \|_2} \to 0 $$
for any $X_0, Y \in \Ltwo(\hat{\Omega})$ with $\Law(X_0) = m_0$
\end{definition}

\subsection{Stochastic flow of probability measure and conditional law}\label{derivative_sub}

In the context of MFG with common noise, we will be dealing with a stochastic flow of probability measure $m=(m_t)_{0\leq t \leq T} \in \CP$, which is  the law conditional on $\tF_t$ of a state process
$$ X_t = \xi_0 + \int_0^t \alpha_t dt + \sigma W_t + \eps \tW_t, \quad \forall t \in [0,T]$$
where $\xi_0 \in \Ltwo_{\F_0}, \alpha \in \Htwo$. To take a derivative using the notion described above at $m_t = \Law(X_t | \tF_t)$, we need to find a random variable to represent such law. The obvious choice is simply the state process $X_t$ itself. As we are dealing with a conditional law, so to do the lifting in an explicit manner, we first need to separate the path space for the individual noise and the common noise. We assume that $(\Omega,\F,\mb{P})$ is in the form $ (\Omega^0 \times \t{\Omega},\F^0 \otimes \tF,\mb{P}^0 \otimes \tP)$ where the individual noise $(W_t)_{0 \leq t \leq T}$ and common noise $(\tW_t)_{0 \leq t \leq T}$ are supported in the space $ (\Omega^0,\F^0 ,\mb{P}^0 )$ and $(\t{\Omega},\tF,\tP)$ respectively. We will also assume that $(\t{\Omega},\tF,\tP)$ is the canonical sample space of the Brownian motion $(\tW_t)_{0 \leq t \leq T}$. To avoid confusion between the lifting space and the original space, we will let $(\h{\Omega}^0,\h{\F}^0,\h{\mb{P}}^0)$ denote a copy of $(\Omega^0,\F^0,\mb{P}^0)$  and $\hY_t \in \Ltwo(\h{\Omega}^0 \times \t{\Omega}; \R)$ denote the copy of $Y_t \in \Ltwo(\Omega; \R)= \Ltwo(\Omega^0 \times \t{\Omega}; \R)$ sharing the same common noise for any random variable $Y_t$. We will use this ``hat" notation throughout the paper when using the derivative with respect to a probability measure in the context of MFG with common noise.

%Consider a continuously differentiable functional $F : \Ptwo \to \R $, we fix a common noise path $\t{\omega} \in \t{\Omega} $ and consider $m_t \in \CP$, the law of $X_t$ conditional on $\tF_t$. Let $U^\eps_t$ be an arbitrary random variable in $(\Omega,\F,\mb{P})$ such that 
%$$\lim_{\eps \to 0} \int_{\Omega^0} (U_t^\eps(\omega^0,\tw))^2 d\mb{P}^0(\omega^0) = 0$$
% for all $t \in [0,T]$, $\t{\mb{P}}$-a.s. Then by definition of the derivative described above applying to a fixed $\tw \in \t{\Omega}$, it follows that
%$$ \lim_{\eps \to 0} \frac{F(\Law(X_t(\cdot,\tw)+U^\eps_t(\cdot,\tw)))-F(\Law(X_t(\cdot,\tw)))-  \h{\mb{E}}\l[\d{m}F(m_t)(\hX_t)\hU^\eps_t \r](\tw)  }{\| U^\eps_t(\cdot,\tw)\|_2 } = 0$$
%where $\hE[\cdot]$ denote the expectation with respect to $\h{\omega} \in \h{\Omega}$ only or equivalently the expectation conditional on $\tF_t$. Observe that $\Law(X_t(\cdot,\tw)) = \Law(X_t|\tF_t)(\tw)$, so we can write
%$$ \lim_{\eps \to 0} \frac{F(\Law(X_t+U^\eps_t | \tF_t))-F(\Law(X_t|\tF_t))-  \h{\mb{E}}\l[\d{m}F(m_t)(\hX_t)\hU^\eps_t \r]  }{\| U^\eps_t(\cdot,\tw)\|_2 } = 0$$
%This type of limit will come out frequently in a subsequent analysis. Again, we would like to note that we use $\hX_t$ and $\hU_t$ instead of $X_t$ and $U_t$ to distinguish the random variables used for lifting and the original random variables.

%cite result instead
%\section{Proof of Theorem \ref{markovian_m}}\label{proof_markov_theta}
%\input{tex/proof_markov_theta}

%cite result instead
%\section{Proof of Theorem \ref{markovian_main}}\label{proof_markov_ue}
%\input{tex/proof_markov_ue}

\bibliographystyle{abbrv}

\end{document}